\documentclass[11pt, twoside]{article}
\usepackage{amsthm}
\usepackage{amscd}
\usepackage{mathpazo}
\usepackage{txfonts}

\usepackage{amsmath}
\usepackage{mathrsfs}
\usepackage{enumerate}

\usepackage{mathtools}
\allowdisplaybreaks[1]
\usepackage{color}
\usepackage{hyperref}
\usepackage{lastpage}
\usepackage{fancyhdr}
\usepackage[T1]{fontenc}
\usepackage{geometry}
\usepackage{graphicx}
\usepackage[dvipsnames]{xcolor}
\allowdisplaybreaks[1]

\numberwithin{equation}{section}
\theoremstyle{plain}
\newtheorem{theorem}{Theorem}[section]        
\newtheorem{corollary}[theorem]{Corollary}             
\newtheorem{lemma}[theorem]{Lemma}

\theoremstyle{definition}
\newtheorem{definition}[theorem]{Definition}                
\theoremstyle{remark}
\newtheorem{remark}[theorem]{Remark}

\newcommand{\bR}{\mathbb{R}}

\newcommand{\W}{\mathbf{W}}
\newcommand{\st}{ \hspace{0.075cm} \hat{\otimes} \hspace{0.075cm}  }

\makeatletter
\@tfor\next:=abcdefghijklmnopqrstuvwxyzABCDEFGHIJKLMNOPQRSTUVWXYZ\do{%
\def\command@factory#1{%
\expandafter\def\csname b#1\endcsname{\mathbf{#1}}
\expandafter\def\csname cl#1\endcsname{\mathcal{#1}}
}
\expandafter\command@factory\next
}

\mathtoolsset{showonlyrefs}

\title{On the Navier-Stokes equation perturbed by rough transport noise}
\author{Martina Hofmanov\'a\thanks{Institute of Mathematics, Technical University of Berlin, Germany, Financial support by the DFG via Research Unit FOR 2402 is gratefully acknowledged.} \thanks{Fakult\"at f\"ur Mathematik, Universit\"at Bielefeld, Postfach 100131, D-33501 Bielefeld}, \;James-Michael Leahy\thanks{Department of Mathematics, University of Southern California, USA}, ];Torstein Nilssen\thanks{Department of Mathematics, University of Oslo, Norway. Funded by the Norwegian Research Council (Project 230448/F20);} \footnotemark[1]}
\date{}

\begin{document}
\maketitle

\begin{abstract}
We consider the Navier-Stokes system in two and three space dimensions perturbed by transport noise and subject to periodic boundary conditions. The  noise  arises from perturbing the advecting velocity field by space-time dependent noise that is  smooth in space and  rough in time. We study the system  within the framework of rough path theory and, in particular, the recently developed theory of unbounded rough drivers. We introduce an intrinsic notion of a weak solution of the Navier-Stokes system, establish suitable a priori estimates and prove existence. In two dimensions, we prove that the solution is unique and stable with respect to the driving noise. 

\bigskip

MSC Classification Numbers: 60H15, 76D05, 47J30, 60H05, 35A15.


\end{abstract}

\section{Introduction}
\label{s:intro}

The theory of  rough paths, introduced by Terry Lyons in his seminal work \cite{Ly98},   can be briefly described as an extension of the classical theory of controlled differential equations that is robust enough to allow for a pathwise (i.e., deterministic) treatment of stochastic differential equations (SDEs).  Since its introduction,  the theory of ordinary and partial differential equations driven by rough signals has progressed substantially. We refer the reader to the works of Friz et al. \cite{CaFr09,CaFrOb11}, Gubinelli et al. \cite{GuTi10,DeGuTi12, GuLeTi06}, Gubinelli--Imkeller--Perkowski \cite{GuImPe15}, Hairer \cite{Ha14} for a  sample of the   literature on the growing subject. In spite of these exciting developments, many PDE  methods have not yet found their rough path analogues. For instance, until recently, it was not known how to  construct  (weak) solutions to rough partial differential equations (RPDEs) using energy methods (or variational methods).

The first results on  energy methods for RPDEs were established in \cite{BaGu15, DeGuHoTi16, 2017arXiv170707470H}. In \cite{BaGu15}, the foundation of the theory of unbounded rough drivers was established and then used to derive  the well-posedness of a linear transport equation driven by a rough-path in the Sobolev scale. Expanding upon the scope of the theory, the authors of \cite{DeGuHoTi16} developed a rough version of Gronwall's lemma and proved the well-posedness of nonlinear scalar conservation laws with rough flux. In the framework of unbounded rough drivers, one can define an intrinsic notion of a weak solution of an RPDE that is equivalent to the usual definition if the driving path is smooth in time. Additionally, one can  obtain an energy estimate of the solution.  Prior to the development of the theory of unbounded rough drivers and  rough Gronwall lemma, these problems remained open. In particular, how to study the well-posedness of the Navier-Stokes system with rough transport noise was out of reach. Most recently, the theory of unbounded rough drivers has been applied to prove the existence, uniqueness and stability of two classes of equations: 1) linear parabolic PDEs with a bounded and measurable diffusion coefficient driven by rough paths \cite{2017arXiv170707470H} and 2) reflected rough differential equations \cite{DeGuHoTi16a}.

The aim of our efforts is to study  the Navier-Stokes system subject to rough transport noise.
We study the system of equations that govern the  evolution of the velocity field $u: \bR_+ \times \bT^d \rightarrow \bR^d$ and the pressure $p : \bR_+ \times \bT^d \rightarrow \bR$ of an incompressible viscous fluid on the $d$-dimensional torus $\bT^d$ perturbed by transport-type noise:
\begin{equation} \label{eq:classicalForm}
\begin{aligned}
\partial_t u +(u-\dot{a})\cdot \nabla u +\nabla p& =  \nu \Delta u,  \\
\nabla \cdot u & = 0,  \\
u(0) & = u_0 \in L^2(\bT^d;\bR^d),  
\end{aligned}
\end{equation}
where $\nu > 0$ is the viscosity coefficient and $\dot{a}$ is the (formal) derivative in time  of a function  $a=a_t(x):\bR_+\times \bT^d\rightarrow \bR^d$ that is divergence-free in space  and has finite $p$-variation in time  for some $p \in [2,3).$ For example, $\dot{a}$ may represent noise that is white in time and colored in space. Such noise  is a formal time derivative of an $L^2(\bT^d)$-valued Wiener process. However, one of the main advantages of the theory of rough paths is that  drivers that are not necessarily martingales or of finite-variation can be considered, which is in direct contrast to the classical semimartingale theory. Consequently, $\dot{a}$ may represent the time derivative of a more general spatially dependent Gaussian or Markov process, such as a fractional Brownian motion, $B^H := (B^{H,1}, \dots, B^{H,K})$ with Hurst parameter $H\in(\frac{1}{3},\frac{1}{2}]$, coupled with a family of vector fields $\sigma = (\sigma_1, \dots ,  \sigma_K):\bT^d\rightarrow \bR^{K\times d}$; that is, for $(t,x)\in \bR_+\times \bT^d$,
$$
a_t(x) =  \sum_{k=1}^{K} \sigma_k(x) B^{H,k}_t.
$$

Even in the case of the unperturbed Navier-Stokes system, it is unknown whether there exists global smooth solutions, and so we study the perturbed system integrated in time and tested against a smooth test function in space. In particular, it is necessary to make sense of the time integral
$\int_0^t (\dot{a}_s\cdot \nabla) u_s \, ds$
as a spatial distribution. Testing this integral against a  smooth function $\phi:\bT^d\rightarrow \bR^d$, we get
\begin{equation}\label{eq:time integral}
\int_0^t (\dot{a}_s\cdot \nabla) u_s \, ds (  \phi)  = - \int_0^t  u_s ( (\dot{a}_s\cdot \nabla  )\phi)\,ds,
\end{equation}
where we have used the divergence-free assumption  $\nabla\cdot \dot{a}=0$.
However, the time integral is not a priori well-defined since we expect the solution $u$ to inherit the same regularity in time  as  $a$ (i.e.,   $p$-variation). Indeed, L.C.\ Young's theorem in \cite{young1936inequality} says that a Riemann-Stieltjes integral $\int fdg$ exists  if there are $p$ and $q$ with $p^{-1}+q^{-1}>1$, such that $f$ is of $p$-variation and $g$ is of $q$-variation. Furthermore, a counterexample is given for  the case $p^{-1}+q^{-1}=1$, and hence  the theorem of Young cannot be used to define \eqref{eq:time integral}, unless  $a$ has $p$-variation in time for  $p \in [1,2)$.

The rough path theory  of Lyons \cite{Ly98} enables us to define the integral \eqref{eq:time integral}, provided that we possess additional information about the driving path, namely its iterated integral, and the integrand. The idea is to iterate the equation for $u$ into the noise integral \eqref{eq:time integral} enough times so that the remainder is regular enough in time to be negligible.  In the case of transport noise, this iteration  leads to an iteration of the spatial derivative. For simplicity, let us explain how this iteration works for the pure transport equation
\begin{equation}\label{eq:puretransport}
\partial_t u=  (\dot{a}\cdot \nabla) u .
\end{equation}
Integrating  \eqref{eq:puretransport} in time, testing against a smooth function $\phi:\bT^d\rightarrow \bR^d$, and then iterating the equation \eqref{eq:puretransport} into itself yields
\begin{align}
u_t (\phi)  & =  u_s(\phi) -  \int_s^t  u_r ( (\dot{a}_r\cdot \nabla) \phi )\,dr \notag \\
& =  u_s(\phi) - u_s \left(  \int_s^t (\dot{a}_{r}\cdot\nabla)\phi\,dr \right)   +\int_s^t \int_s^{r_1}  u_{r_2} \left( (\dot{a}_{r_2}\cdot \nabla)(\dot{a}_{r_1}\cdot \nabla) \phi \right) \,\, dr_2\, dr_1 \notag\\
& =  u_s(\phi) - u_s \left(  \int_s^t (\dot{a}_{r}\cdot\nabla)\phi\,dr \right)  +  u_s \left( \int_s^t \int_s^{r_1} (\dot{a}_{r_2}\cdot \nabla)(\dot{a}_{r_1}\cdot \nabla)\phi \,\,\, dr_2\,\, dr_1 \right)  \notag \\
&\quad\quad  - \int_s^t \int_s^{r_1} \int_s^{r_2}  u_{r_3} \left( (\dot{a}_{r_3}\cdot \nabla)(\dot{a}_{r_2}\cdot \nabla)(\dot{a}_{r_1}\cdot \nabla)\phi\right) \,dr_3\, dr_2\, dr_1 , \label{three deriv}
\end{align}
where we have used the divergence-free assumption  $\nabla\cdot \dot{a}=0$. If we define the operators
\begin{equation} \label{eq:URDIntroDef}
A_{st}^{1} \phi =   \int_s^t  (\dot{a}_{r}\cdot \nabla) dr\phi \quad \textnormal{ and } \quad  A_{st}^{2} \phi =  \int_s^t \int_s^{r_1} (\dot{a}_{r_2}\cdot \nabla)(\dot{a}_{r_1}\cdot \nabla) dr_2dr_1 \phi,
\end{equation}
and let $\delta u_{st}=u_t-u_s$, then solving the transport equation \eqref{eq:puretransport}   corresponds to finding a map $t \mapsto u_t$ such that  $u^{\natural}$ defined by
\begin{equation}\label{eq:iteratedpuretransport}
u^{\natural}_{st} (\phi):= \delta u_{st} (\phi) - u_s \left( [A_{st}^{1,*}+A_{st}^{2,*}] \phi \right) 
\end{equation}
is of order $o(|t-s|)$, and hence is negligible. That is, the expansion $[A_{st}^1+A_{st}^2] u_s$ tested against $\phi$,  provides a good local approximation of the time integral \eqref{eq:time integral}, which is  uniquely defined by the sewing lemma (see Lemma \ref{sewingLemma}). Notice that if $a$ is smooth in time and space, then
\eqref{eq:iteratedpuretransport} is an equivalent formulation of the transport equation \eqref{eq:puretransport}. 
Because the time singularities in \eqref{eq:URDIntroDef} are smoothed out by averaging over time, 
the equation \eqref{eq:iteratedpuretransport} does not contain any time derivatives, and hence the formulation  
is well-suited for irregular drivers. Under certain conditions, the pair  $\mathbf{A}=(A^1,A^2)$ defines an \textit{unbounded rough driver} as defined in \cite{BaGu15} and in Section \ref{sec:unboundedrough} below.

In order to show that the remainder $u^\natural$ is of order $o(|t-s|)$, we shall regard it as a distribution of third order with respect to the space variable; note that three derivatives are taken in \eqref{three deriv}. One of the key aspects of the theory of unbounded rough drivers is the process by which one obtains a priori  estimates of the remainder $u^\natural$  (see Section \ref{sec: a priori}). The technique involves  obtaining estimates of $\delta u^{\natural}_{s\theta t}:=u^{\natural}_{st}-u^{\natural}_{s\theta}-u^{\natural}_{\theta t}$, interpolating between time and space regularity of various terms, and applying the sewing lemma  (i.e., Lemma \ref{sewingLemma}). This is yet another example of the trade-off between time and space regularity pertinent to many PDE problems. Notice that if $a$ is $\alpha$-H\"older continuous (essentially equivalent to $\alpha^{-1}$-variation) with respect to the time variable and the solution $u$ has the same regularity in time, then the first two terms on the right-hand-side of \eqref{eq:iteratedpuretransport} are proportional to $|t-s|^\alpha$ and the  last term on the right-hand-side can be bounded by $|t-s|^{2\alpha}$. Thus, in the case  $\alpha\in (\frac{1}{3},\frac{1}{2}]$, there has to be a cancellation between the terms on the right-hand-side to guarantee that $u^\natural$ is  of order $o(|t-s|)$. On the other hand, the right-hand-side of \eqref{eq:iteratedpuretransport}  is a distribution of second order with respect to the space variable. Accordingly, the necessary improvement of time regularity can be obtained at the cost of loss of space regularity; that is, considering $u^\natural$ rather as a distribution of third order.

In this paper, we assume that  the noise term $a$ can be factorized as follows:
\begin{equation}\label{eq:aa3}
a_t(x) = \sigma_k(x) z^k_t=  \sum_{k=1}^{K} \sigma_k(x) z^k_t,
\end{equation}
where we adopt the convention of summation over repeated indices $k\in \{1,\ldots,K\}$ here and below. We also assume that for all $k\in \{1,\ldots,K\}$, the vector fields $\sigma_k:\bT^d\rightarrow \bR^d$ are bounded, divergence-free, and twice-differentiable with bounded first and second derivatives. The driving signal $z$ is assumed to be a $\bR^{K}$-valued path of finite $p$-variation for some $p \in [2,3)$ that can be lifted to a geometric rough path $\bZ=(Z,\mathbb{Z}).$ The first  component of $\bZ$ is the increment of $z$ (i.e., $Z_{st}=z_t-z_s$) and the second component is the so-called L\'evy's area, which plays the role of the iterated integral $\mathbb{Z}_{st}=:\int_{s}^t\int_{s}^r\, dz_{r_1}\otimes \, dz_r$. In the smooth setting, the iterated integral can be defined as a Riemann integral, whereas in the rough setting, it has to be given as an input datum;  the two-index map $\mathbb{Z}_{st}$ is  assumed to satisfy Chen's relation 
$$
\delta \mathbb{Z}_{s\theta t}:=\mathbb{Z}_{s t}-\mathbb{Z}_{s \theta}-\mathbb{Z}_{\theta t} = Z_{s\theta }\otimes Z_{\theta t }, \;\;s\le \theta \le t,
$$
and to be two-times as regular in time as the path $z$.
For instance, if $z$ is a Wiener process, then an iterated integral can be constructed using the Stratonovich stochastic integration. Nevertheless, many other important stochastic processes give rise to  (two-step) rough paths. For more details, we refer the reader to Section \ref{ss:rp} and the literature mentioned therein.

The motivation for a perturbation  of the form $-\dot{a}\cdot \nabla u$ comes from the modeling of a turbulent flow of a viscous fluid.  In the Lagrangian formulation, an incompressible fluids evolution is traditionally specified in terms of the flow map of particles initially at $X$:
$$\dot{\eta}_t(X)=u_t(\eta_t(X)),  \;\;\eta(X,0)=X\in \bT^d, \quad  \nabla \cdot u= 0.$$
If we assume the associated fluid flow map is a composition of a mean flow depending on slow time $t$ and a rapidly fluctuating flow with fast time scales $\epsilon^{-1}t$, $\epsilon \ll 1$, then provided that the fast-dynamics are sufficiently chaotic,  on  time-scales of order $\epsilon^{-2}$, the averaged slow-dynamics are described by the SDE \cite{cotter2017stochastic}
\begin{equation}
d\bar{\eta}_t(X)=\bar{u}_t(\bar{\eta}_t(X))dt-\sigma_k(\bar{\eta}(X,t))\circ dw^k_t, \;\;\bar{\eta}(X,0)=X\in \bT^d, \;\;  \nabla \cdot \bar{u}= 0, \;\; \nabla \cdot \sigma_k= 0, \label{eq:SFlow}
\end{equation}
where  $w:=\{w^k\}_{k=1}^{\infty}$ is a sequence of independent Brownian motions and the stochastic integral is understood in the Stratonovich sense. The flow dynamics given by \eqref{eq:SFlow} encompasses models of stochastic passive scalar turbulence that were originally proposed by R.\ Kraichnan \cite{kraichnan1968small} and  further developed in \cite{gawedzki1996university, gawedzki2000phase} and other works. In \cite{brzezniak1991stochastic, brzezniak1992stochastic, mikulevicius2004stochastic, mikulevicius2005global}, it was shown that the system of equations governing the resolved scale velocity  field $\bar{u}$ and pressure $p$ and $\{q_k\}_{k=1}^{\infty}$ is a  stochastic version of the Navier-Stokes system with transport noise:
\begin{equation}
d\bar{u}+(\bar{u}dt-\sigma_k\circ dw^k_t)\cdot \nabla \bar{u}+\nabla  p dt+\nabla q_k\circ dw^k_t=\nu \Delta \bar{u} dt\label{eq:SNS1}.
\end{equation}
The existence and uniqueness of solutions of \eqref{eq:SNS1} has been well-studied \cite{brzezniak1992stochastic, flandoli1995martingale,mikulevicius2004stochastic, mikulevicius2005global}. In \cite{mikulevicius2005global}, the authors proved the existence of global weak-probabilistic solutions (i.e., martingale solutions) of a general class of stochastic Navier-Stokes equations on the whole space, which included   \eqref{eq:SNS1}. Moreover, in dimension two, the uniqueness of the global strong probabilistic solution was established in \cite{mikulevicius2005global} as well. The existence of strong global solutions for the stochastic Navier-Stokes system \eqref{eq:SNS1} in three-dimensions  is still an open problem.

In this paper, we  develop a (rough) pathwise solution theory for \eqref{eq:classicalForm}, which, in particular, offers a pathwise interpretation of  \eqref{eq:SNS1} for $k\in \{1,\ldots,K\}$. We establish the existence of weak solutions in two and three space dimensions (see Theorem \ref{existenceThmNoGalerkin}) by establishing energy estimates, including the recovery of the pressure (see Section~\ref{s:pressure}). To prove existence, we use Galerkin approximation combined with a suitable mollification of the driving signal, uniform energy estimates of the solution and the remainder terms and a compactness argument. In addition, in two space dimensions and for constant vector fields $\sigma_k$, we prove uniqueness and pathwise stability with respect to the given driver and initial datum via a tensorization argument (see Theorem \ref{contractiveTheoremNoContraction} and Corollary \ref{cor:stability}). This result implies a Wong-Zakai approximation theorem  for the Wiener driven SPDE  \eqref{eq:SNS1}.
To the best of our knowledge, this is the first Wong-Zakai type result for the Navier-Stokes system \eqref{eq:SNS1}.  
There are a substantial number of Wong-Zakai results for infinite dimensional stochastic evolution equations in various settings. We mention only the work \cite{CM11} of A. Millet and I. Chueshov  in which the authors derive a Wong-Zakai result and support theorem for a general class of stochastic 2D hydrodynamical systems, including 2D stochastic Navier-Stokes. However, the diffusion coefficients in \cite{CM11} are assumed to have linear growth on $L^2(\bT^2; \bR^2)$, and hence do not cover transport noise. We do note, however, that in  \cite{CM10},  A.\ Millet and I.\ Chueshov establish a  large deviation result for stochastic 2D hydrodynamical systems that does hold true for transport noise. 

Our approach relies on a suitable formulation of the system \eqref{eq:classicalForm} that is similar to the formulation of the pure-transport equation \eqref{eq:iteratedpuretransport} discussed above. However, due to the  structure of \eqref{eq:classicalForm} and the fact that a solution is the pairing of a velocity field and pressure $(u,p)$, the formulation is more subtle. In fact, we present two equivalent (rough) formulations of \eqref{eq:classicalForm} in Section \ref{ss:form}.

Let $P$ be the Helmholtz-Leray projection and $Q=I-P$ (see Section 2.1 for more details). Applying $P$ and $Q$ separately to \eqref{eq:classicalForm},
we obtain the system of coupled equations	
\begin{align*}
\partial_t u +P[(u\cdot \nabla) u] &  =  \nu \Delta u+P[(\dot{a}\cdot \nabla) u]  \notag \\
Q[(u\cdot \nabla) u]+\nabla p & =Q[(\dot{a}\cdot \nabla) u]\label{Pressure eq}. \notag 
\end{align*}
We can then perform an iteration of the equation for $u$ in the time integral of $P[\dot{a}\cdot \nabla u]$ and $Q[\dot{a}\cdot \nabla u]$ like we illustrated above for the pure transport equation \eqref{eq:iteratedpuretransport}. After doing so, we obtain  a coupled system of equations for the velocity field and pressure for which the associated unbounded rough drivers are intertwined and a  version of the so-called Chen's relation holds true (see \eqref{quasiChen} and Definition \ref{def:solution2}). We derive a second equivalent formulation by summing the coupled equations from the first formulation. This second formulation is a single equation for the velocity field in which  a modified Chen's relation holds (see \eqref{quasiChen2} and Definition \ref{def:solution}). An alternative way to arrive at the second formulation is by iterating \eqref{eq:classicalForm} and using that $\nabla p=Q[(\dot{a}\cdot \nabla) u]-Q[(u\cdot \nabla) u]$. 

The presentation of this paper is organized as follows. In Section \ref{s:setting}, we define our notion of solution and state our main results. In Section \ref{sec: a priori}, we derive a priori estimates of remainder terms, which are used in   Section \ref{main} to prove our main results. Several auxiliary results that are used to prove the main results are presented in the appendix.

\section{Mathematical framework and main results}
\label{s:setting}

\subsection{Notation and definitions}
\label{ss:notation}

We begin by fixing the notation that we use throughout the paper. 

We shall write $a \lesssim b$ if   there exists a positive constant $C$ such that $a \le b$. If the contant $C$ depends only on the parameters $p_1,\ldots,p_n$, we shall also write $C = C(p_1, \ldots, p_n)$ and $\lesssim_{p_1,\ldots,p_n}$.

Let $\bN_0=\bN\cup \{0\}$. For  a given $d\in \mathbf{N}$, let  $\bT^d=\bR^d/(2\pi \bZ)^d$ be the $d$-dimensional flat torus and denote by $\, dx$ the unormalized Lebesgue measure on $\bT^d$. As usual, we blur the distinction between periodic functions and functions defined on the torus $\bT^d$. For a given Banach space $V$ with norm $|\cdot|_V$, we denote by  $\clB(V)$ the Borel sigma-algebra of $V$ and by $V^*$ the continuous dual of $V$.
For given Banach spaces $V_1$ and $V_2$, we denote by $\clL(V_1,V_2)$  the space of continuous linear operators from $V_1$ to $V_2$ with the operator norm denoted by $|\cdot|_{\clL(V_1,V_2)}$.

For a given sigma-finite measured space $(X,\clX,\mu)$, separable Banach space $V$ with norm $|\cdot|_V$, and $p\in [1,\infty]$, we denote by $L^p(X;V)$  the  Banach space of all $\mu$-equivalence-classes of strongly-measurable functions $f: X\rightarrow V$ such that 
$$
|f|_{L^p(X;V)}:=\left(\int_{X}|f|_V^p\,d\mu\right)^{\frac{1}{p}}<\infty,
$$
equipped with the norm 
$
|\cdot |_{L^p(X;V)}.
$
We denote by  $L^{\infty}(X;V)$  the Banach space of all $\mu$-equivalence-classes of strongly-measurable functions $f: X\rightarrow V$ such that 
$$
|f |_{L^{\infty}(X;V)}:=\operatorname{esssup}_X|f|_V:=\inf\{a\in \mathbf{R}: \mu(|f|_V^{-1}((a,\infty))=0)\}<\infty,
$$
where $|f|_V^{-1}((a,\infty))$ denotes the preimage of the set $(a,\infty)$ under the map $|f|_V : X\rightarrow \mathbf{R}$, equipped with the norm 
$
|\cdot |_{L^{\infty}(X;V)}.
$
It is well-known that if $V=H$ is a  Hilbert space with inner product $(\cdot,\cdot)_H$, then $L^2(X;H)$ is a Hilbert space equipped with the inner product
$$
(f,g)_{L^2(X;H)}=\int_{X}(f,g)_H\,d\mu, \quad f,g\in L^2(X;H).
$$
For a given Hilbert space $H$, we let
$
L^2_TH=L^2([0,T];H)$ and $   L^{\infty}_TH=L^\infty([0,T];H).
$
Moreover, let $\bL^2=L^2( \bT^d ; \bR^d)$.

For a given Hilbert space $V$, and real number $T>0$, we let $C_TH=C([0,T];H)$ denote the Banach space  of continuous functions from $[0,T]$ to $G$, endowed with the supremum norm in time.

For a given $n\in \bZ^d$, let $e_n: \bT^d\rightarrow \bC$ be defined by $e_n(x)=(2\pi)^{-\frac{d}{2}}e^{in\cdot x}$. It is well-known that $\{e_n\}_{n\in \bZ^d}$ constitutes an orthonormal system of $L^2(\bT^d;\bC)$, and hence for all $f,g\in \bL^2$,
$$
f=\sum_{n\in\bZ^d}\hat{f}_ne_n,\quad (f,g)_{\bL^2}=\sum_{n\in \bZ^d} \hat{f}_n\cdot \overline{\hat{g}_n},
$$ 
where for each $n\in \bZ^d$,
$$\hat{f}^i_n=\int_{\bT^d}f^i(x)e_{-n}(x)\,\, dx, \quad i\in \{1,\ldots,d\}.$$
Let $\clS$ be the Fr\'echet space of infinitely differentiable periodic complex-valued functions with the usual set of seminorms.
Let $\clS'$ be the continuous dual space of $\clS$ endowed with the weak-star topology. For a given $\Lambda \in \clS'$ and test function $\phi \in \clS$, we denote by $\Lambda(\phi)$ the value of a distribution  $\Lambda$ at $\phi \in \clS$. Since $e_n\in \clS$, for a given $f \in \clS'$ and $n\in \bZ^d$, we define $\hat{f}_n=f(e_n).$  It is well-known that
$
f=\sum_{n\in \mathbf{Z}^d}\hat{f}_ne_n,
$ 
where  convergence  holds in $\clS$ if $f\in \clS$ and in $\clS'$ if $f\in \clS'$. This extends trivially to the set $\bS'=(\clS')^d$ of continuous linear functions from $\bS=(\clS)^d$ to $\bC$ endowed with the weak-star topology.

For a given $\alpha\in \bR$, we denote by $\W^{\alpha,2}$ the Hilbert space
$$
\W^{\alpha,2}=(I-\Delta)^{-\frac{\alpha}{2}}\bL^2=\{f\in \bS': (I-\Delta)^{\frac{\alpha}{2}}f\in \bL^2\}
$$
with inner product
$$
(f,g)_{\alpha}=((I-\Delta)^{\frac{\alpha}{2}}f,(I-\Delta)^{\frac{\alpha}{2}}g)_{\bL^2}=\sum_{n\in \bZ^d}(1+|n|^2)^\alpha\hat{f}_n\overline{\cdot \hat{g}_n},\quad  \; f,g\in \W^{\alpha,2}
$$
and  induced norm $|\cdot |_{\alpha}$. For notational simplicity, when $m=0$ we omit the index in the inner product, i.e. $(\cdot, \cdot) := (\cdot, \cdot)_0$.  Moreover,  for any $u\in \bW^{1,2}$, we write
$
|\nabla u|_0^2=\sum_{i=1}^d|D_iu|_0^2.
$
It is easy to see that $\W^{\alpha,2}\subset \W^{\beta,2}$ for $\alpha,\beta\in \bR$ with  $\alpha>\beta$ and that $\bS$ is dense in $\W^{\alpha,2}$ for all $\alpha\in \bR$.  It can be shown that  for all $\alpha,\beta\in \bR$, the map $i_{\alpha-\beta,\alpha+\beta}: \bW^{\alpha-\beta,2}\rightarrow (\bW^{\alpha+\beta,2})^*$ defined by
$$
i_{\alpha-\beta,\alpha+\beta}(g)(f)= \langle g,f\rangle_{\alpha-\beta,\alpha+\beta}:=((I-\Delta)^{\frac{-\beta}{2}}g ,(I-\Delta)^{\frac{\beta}{2}}f)_{\alpha}, 
$$
for all $f\in \bW^{\alpha+\beta,2}$ and $ g\in \bW^{\alpha-\beta,2},$ is an isometric isomorphism. 

Let
$$
\bH^0=\left\{f\in \W^{0,2}: \;\nabla \cdot  f= 0\right\}=\left\{f\in \W^{0,2}: \hat{f}_n\cdot n=0, \;\forall n\in \bZ^d\right\}.
$$
We define $P :\bS'\rightarrow \bS'$  by
$$
Pf=\sum_{n\in\mathbf{Z}^d}\left(\hat{f}_n-\frac{n\cdot \hat{f}_n}{|n|^2}n\right)e_n, \quad f\in \bL^2,
$$
and let  $Q = I - P$. It follows that  $P$ is a projection of $\bL^2$ onto $\bH^0=P\bL^2$ and that $\bL^2$ possesses the orthogonal decomposition 
$$
\bL^2=P\bL^2\oplus Q\bL^2.
$$ 
Moreover, it is clear that  $P,Q\in \clL(\W^{\alpha,2},\W^{\alpha,2})$ and that $P$ and $Q$  have operator norm less than or equal to  one for all $\alpha\in \bR$.  We set $$\bH^{\alpha}=P\W^{\alpha,2} \;\;\&\;\; \bH_{\perp}^{\alpha}=Q\W^{\alpha,2}.$$  It can be shown that  for all $\alpha\in \bR$ (see Lemma 3.7 in \cite{mikulevicius2003cauchy}),
$$
\W^{\alpha,2} =\bH^\alpha\oplus \bH_{\perp}^\alpha,
$$
where \begin{equation}\label{direct sum prop}
\langle f,g\rangle_{-\alpha,\alpha}=0, \quad \forall g\in \bH_{\perp}^\alpha, \quad \forall f\in \bH^{-\alpha},\end{equation} and 
$$
\bH^\alpha=\left\{f\in \W^{\alpha,2}: \;\nabla \cdot  f= 0\right\},
$$
$$
\bH_{\perp}^\alpha=\{g\in \W^{\alpha,2}:  \langle f,g\rangle_{-\alpha,\alpha}=0, \;\; \forall f\in \bH^{-\alpha} \}.
$$
Using \eqref{direct sum prop}, one can check that  $i_{-\alpha,\alpha}:\bH^{-\alpha}\rightarrow (\bH^{\alpha})^*$ and $i_{-\alpha,\alpha}:\bH^{-\alpha}_{\perp}\rightarrow (\bH^{\alpha}_{\perp})^*$ are isometric isomorphisms  for all $\alpha\in \bR$. 

For each vector $n\in \mathbf{Z}^d $, $n \neq 0$, we can find  $(d-1)$-vectors $\{m_1(n),\cdots,m_{d-1}(n)\}\subseteq \bR^d$ that are  of unit length and orthogonal to both $-n$ and $n$ with respect to the standard inner product on $\bR^d$. It follows that $m_j(n)=m_j(-n)$ for all $j\in \{1,\ldots,d-1\}$ and $n\in \bZ^d$ such that $n \neq 0$, and that  $\mathbf{f}_{j,n}=m_j(n)e_n$, $n\in \bZ^d$, $n \neq 0$, $j\in \{1,\ldots,d-1\},$ is an orthonormal family of $\{u\in L^2(\bT^d;\bC^d): \nabla \cdot u=0\}$. 
In particular, for $n=[n_1, n_2]^T\in \mathbf{Z}^2$, $n \neq 0$ , the vector $\frac{n^{\perp}}{|n|}=[\frac{n_2}{|n|}, -\frac{n_1}{|n|}]^T$ is orthogonal to $n$ and $-n$ and has length one. Thus, if $d=2$, $\{\mathbf{f}_n=\frac{n^{\perp}}{|n|}e_n$, $n\in \bZ^2$, $n \neq 0\}$ is an orthonormal family $L^2(\bT^d;\bC)$. Notice that in this case
\begin{equation}\label{basis prop}
\nabla^T \mathbf{f}_n= \frac{n^{\perp}}{|n|}  in^Te_n= -\frac{n^{\perp}}{|n|}  (-in^T)e_n=-\nabla^T \mathbf{e}_{-n}=-\nabla^T \overline{\mathbf{f}_{n}}.
\end{equation}
One can then check that $\bw_{j,n}^{sin}(x):=\frac{\sqrt{2}}{(2\pi)^{\frac{d}{2}}}m_j(n)\sin(n\cdot x)$ and $\bw_{j,n}^{cos}(x):=\frac{\sqrt{2}}{(2\pi)^{\frac{d}{2}}}m_j(n)\cos(n\cdot x)$, for
$j\in \{1,\ldots,d-1\}$  and $n\in \bZ^d$ such that $n_1\ge 0$,
form an orthonormal basis of $\bH^0$ and an orthogonal basis of $\bH^1$.
Moreover, they are the eigenfunctions of the Stokes operator $P\Delta$ on $\bH^0$ with corresponding eigenvalues  $\lambda_{j,n}=|n|^2$. We re-index the sequence of eigenfunctions and eigenvalues by $\{h_n\}_{n=1}^{\infty}$ and $\{\lambda_n\}_{n=1}^{\infty}$, respectively, where $\{\lambda_n\}_{n=1}^{\infty}$ is a strictly positive increasing sequence tending to infinity. 

The following considerations shall enlighten the construction of the  unbounded rough drivers associated with  \eqref{eq:classicalForm} (see Section \ref{ss:form}). Let $\sigma:\bT^d\rightarrow \bR^d$ be  twice differentiable and divergence-free. Moreover, assume that the derivatives of $\sigma$ up to order two are bounded uniformly by a constant $N_0$.
Let  $\clA^1=\sigma\cdot \nabla=\sum_{i=1}^d\sigma^iD_i$ 
and  $\clA^2=(\sigma\cdot \nabla)(\sigma\cdot \nabla).$
It follows  that there is a  constant  $N=N(d,N_0,\alpha)$  such that 
$$
|\clA^1|_{\clL(\W^{\alpha+1,2},\W^{\alpha,2})}\le N, \;\forall \alpha\in [0,2],\quad
|\clA^2 f|_{\clL(\W^{\alpha+2,2},\W^{\alpha,2})}\le N, \;\forall \alpha\in [0,1].
$$
We refer the reader to \cite{mikulevicius2001note} for the estimates in the fractional norms; the estimates given in \cite{mikulevicius2001note} are on the whole space, but can easily be adapted to the periodic setting.
Since  $P\in \clL(\W^{\alpha,2},\bH^\alpha)$ and $Q\in \clL(\W^{\alpha,2}, \bH^\alpha_{\perp})$ for all $\alpha\in \bR$,  both of which have operator norm bounded by $1$, we have 
\begin{equation} \label{A1bound}
|P \clA^1|_{\clL(\bH^{\alpha+1},\bH^\alpha)}\le N, \quad |Q \clA^1|_{\clL(\bH_{\perp}^{\alpha+1},\bH_{\perp}^{\alpha})}\le N,\;\forall \alpha\in [0,2],
\end{equation}
and
\begin{equation} \label{A2bound}
|P \clA^2|_{\clL(\bH^{\alpha+2},\bH^\alpha)}\le N, \quad |Q \clA^2|_{\clL(\bH_{\perp}^{\alpha+2},\bH_{\perp}^{\alpha})}\le N,\;\forall \alpha\in [0,1],
\end{equation}
and hence $(P \clA^1)^*\in \clL((\bH^{\alpha})^*, (\bH^{\alpha+1})^*)$ and $(Q \clA^1)^*\in \clL((\bH_{\perp}^{\alpha})^*,(\bH^{\alpha+1}_{\perp})^*)$ for $\alpha\in [0,2]$ and  $(P \clA^2)^*\in \clL((\bH^{\alpha})^*, (\bH^{\alpha+2})^*)$ and $(Q \clA^2)^*\in \clL((\bH_{\perp}^{\alpha})^*,(\bH^{\alpha+2}_{\perp})^*)$ for $\alpha\in [0,1]$.
Making use of the divergence-free property of $\sigma_k$, $k\in \{1,\ldots,K\}$, we find 
$$
((-P\clA^1) f, g)=(f,P\clA^1 g),\quad \forall f,g\in \bS \cap \bH^0,
$$
and
$$
((-Q\clA^1) f, g)=(f,Q\clA^1 g),\quad \forall f,g\in \bS \cap \bH^0_{\perp},
$$
which implies that $(-P\clA^1 )^*=P\clA^1$ and $(-Q\clA^1)^*=Q\clA^1$.
Thus, owing to the characterization of the duality between $\W^{\alpha ,2}$ and $\W^{-\alpha,2}$ through the $\bL^2$ inner product, we have    $$P \clA^1\in \clL(\bH^{-\alpha}, \bH^{-(\alpha+1)}), \;\;Q \clA^1\in \clL(\bH_{\perp}^{-\alpha}, \bH^{-(\alpha+1)}_{\perp}),
$$
$$P \clA^2\in \clL(\bH^{-\alpha}, \bH^{-(\alpha+2)}), \; \;Q \clA^2\in \clL(\bH_{\perp}^{-\alpha}, \bH^{-(\alpha+2)}_{\perp}).$$

\medskip

In order to analyze the convective term, we employ the classical notation and bounds.
Owing to Lemma 2.1 in \cite{RT83} adapted to fractional norms (see \cite{hebey2000nonlinear}), the trilinear form
$$
b(u,\varv,w)= \int_{\bT^d} ((u\cdot \nabla)\varv)\cdot w \, \, dx=\sum_{i,j=1}^d\int_{\bT^d} u^i D_i\varv^j w^j \,\, dx
$$
is  continuous on $\bW^{\alpha_1,2}\times \bW^{\alpha_2 + 1,2}\times \bW^{\alpha_3,2}$ if $\alpha_1,\alpha_2,\alpha_3\in \bR_+$  
satisfy
\begin{gather*}
\alpha_1 + \alpha_2 + \alpha_3 \geq \frac{d}{2}, \;\; \textnormal{ if }\;\; \alpha_i \neq \frac{d}{2} \;\textnormal{ for all } i\in \{1,2,3\},\\\alpha_1 + \alpha_2 + \alpha_3 > \frac{d}{2}, \;\; \textnormal{ if } \alpha_i = \frac{d}{2}\; \textnormal{ for some }  i\in \{1,2,3\};
\end{gather*}
that is,
\begin{equation}\label{trilinear form estimate}
b(u,v,w)\lesssim_{\alpha_1,\alpha_2,\alpha_3,d}
|u|_{\alpha_1}|v|_{\alpha_2+1}|w|_{\alpha_3}.
\end{equation}
In the case $d=2$, by virtue of the  Gagliardo-Nirenberg interpolation inequality $|\phi|_{L^4(\bT^2,\bR^2)}\lesssim|\phi|_0^{\frac{1}{2}} | \phi|_{1}^{\frac{1}{2}}$, we have
\begin{equation}\label{ineq:Lady est}
b(u,\varv,w)\lesssim |u|_0^{\frac{1}{2}}| u|_1^{\frac{1}{2}}|  v|_1|w|_0^{\frac{1}{2}}| w|_1^{\frac{1}{2}}, \quad \forall u,v,w\in \bW^{1,2},
\end{equation}
which plays an important role in the uniqueness proof (see  Theorem \ref{contractiveTheorem}).
Moreover, for all $u\in  \bH^{\alpha_1}$ and  $(\varv,w)\in \bW^{\alpha_2+1,2}\times \bW^{\alpha_3,2}$ such that $\alpha_1,\alpha_2,\alpha_3$ satisfy \eqref{trilinear form estimate}, we have
\begin{equation}\label{eq:B prop}
b(u,\varv,w)=-b(u,w,\varv) \quad \textnormal{and} \quad b(u,\varv,\varv)=0.
\end{equation}
For $\alpha_1,\alpha_2,$ and $\alpha_3$ that satisfy \eqref{trilinear form estimate} and any given $(u,\varv)\in \bW^{\alpha_1,2}\times \bW^{\alpha_2+1,2}$, we define $B(u,\varv)\in \bW^{-\alpha_3,2}$ by
$$
\langle B(u,\varv),w\rangle_{-\alpha_3,\alpha_3}=b(u,\varv,w), \quad \forall w\in \bW^{\alpha_3,2}.
$$
Similarly, we define $B_P=PB$ and $B_Q=QB$ and note that $$B_P:=PB: \bW^{\alpha_1,2}\times \bW^{\alpha_2+1,2}\rightarrow \bH^{-\alpha_3},\quad B_Q:=QB:\bW^{\alpha_1,2}\times \bW^{\alpha_2+1,2}\rightarrow \bH^{-\alpha_3}_{\perp},$$
for $\alpha_1,\alpha_2,$ and $\alpha_3$ that satisfy \eqref{trilinear form estimate}.
We set $$B(u)=B(u,u), \;\; B_P(u) := B_P(u,u), \; \textnormal{and}\; B_Q(u):=B_Q(u,u).$$
\subsection{Smoothing operators}

As in \cite{BaGu15}, we will need a family of smoothing operators $(J^{\eta})_{\eta \in (0,1]}$ acting on the scale of spaces $(\W^{\alpha,2})_{\alpha \in \mathbf{R}}$; that is, we require a family $(J^{\eta})_{\eta \in (0,1]}$ such that for all $\alpha \in  \mathbf{R}$ and $\beta \in \bR_+$,
\begin{equation} \label{smoothingOperator}
|(I - J^{\eta}) f |_{\alpha} \lesssim \eta^{\beta} |f|_{\alpha+ \beta} \hspace{.5cm} \textnormal{ and} \hspace{.5cm} | J^{\eta} f |_{\alpha+\beta} \lesssim \eta^{-\beta} |f|_{\alpha}.
\end{equation}
We construct these operators from the  frequency cut-off operator  $S_N : \bS' \rightarrow\bS$ defined by
$$
S_N f= \sum_{ |n | < N } \hat{f}_n e_n .
$$
It follows that for all $\alpha\in \bR$ and $\beta \in \bR_+$,
$$
|f - S_N f|_{\alpha  }^2 = \sum_{ |n| \geq  N} (1 + |n|^2)^{\alpha} |\hat{f}_n|^2 \leq N^{- 2 \beta} \sum_{ |n| \geq  N} (1 + |n|^2)^{\alpha + \beta } |\hat{f}_n|^2 \leq N^{- 2 \beta} |f|_{\alpha+\beta}^2
$$
and 
$$
|S_N f|_{\alpha + \beta }^2 = \sum_{|n| < N }   (1 + |n|^2)^{\alpha +\beta } |\hat{f}_n|^2 \leq (1+  N^2)^{\beta} \sum_{ |n| \geq  N} (1 + |n|^2)^{\alpha } |\hat{f}_n|^2 \lesssim N^{2 \beta} |f|_{\alpha}^2 .
$$
We define $J^{\eta} := S_{ \lfloor \eta^{-1} \rfloor}$. It is then clear that $J^{\eta}$ is a smoothing operator on $\W^{\alpha,2}$ and that  it leaves the  subspaces $\bH^{\alpha}$ and $\bH_{\perp}^{\alpha}$ invariant.

\subsection{Rough paths} \label{ss:rp}

For a given interval $I,$ we define $\Delta_I := \{(s,t)\in I^2: s\le t\}$ and $\Delta^{(2)}_I := \{(s,\theta,t)\in I^3: s\le\theta\le t\}$.  For a given $T>0$, we let $\Delta_T := \Delta_{[0,T]}$ and $\Delta^{(2)}_T=\Delta^{(2)}_{[0,T]}$ Let $\mathcal{P}(I)$ denote the set of all partitions of an interval $I$ and let $E$ be a Banach space with norm $| \cdot |_E$. 
A function $g: \Delta_I \rightarrow E$ is said to have finite $p$-variation for some $p>0$ on $I$ if 
$$
|g|_{p-\textnormal{var};I;E}:=\sup_{(t_i)\in \clP(I)}\left(\sum_{i}|g _{t_i t_{i+1}}|^p_E\right)^{\frac{1}{p}}<\infty, 
$$
and we denote by $C_2^{p-\textnormal{var}}(I; E)$ the set of all continuous functions with finite $p$-variation on $I$ equipped with the seminorm $|\cdot |_{p-\textnormal{var}; I;E}$. In this section  we drop the dependence of norms  on the space $E$ when convenient. We denote by $C^{p-\textnormal{var}}(I; E)$ the set of all paths $z : I \rightarrow E$ such that $\delta z \in C_2^{p-\textnormal{var}}(I; E)$, where  $\delta z_{st} := z_t - z_s$.

For a given interval $I$, a two-index map $\omega: \Delta_I \rightarrow [0,\infty)$ is called superadditive if for all $(s,\theta,t)\in \Delta^{(2)}_I$,
$$
\omega(s,\theta)+\omega(\theta ,t)\le \omega(s,t).
$$
A two-index map $\omega: \Delta_I \rightarrow [0,\infty)$  is called a control if it is superadditive, continuous on $\Delta_I$ and for all $s\in I$, $\omega(s,s)=0$. 

If for a given $p > 0$, $g \in C^{p-\textnormal{var}}_2(I; E)$, then  it can be shown that the 2-index map $\omega_g: \Delta_I \rightarrow [0,\infty)$ defined by
$$
\omega_g(s,t)= |g|_{p-\textnormal{var};[s,t]}^p 
$$
is a control  (see, e.g.,  Proposition 5.8 in \cite{FrVi10}). It is clear that  $ |g_{st}| \leq \omega_g(s,t)^{\frac{1}{p}}$ for all $(s,t)\in \Delta_I$. If $\omega$ is a control such that $|g_{st}| \leq \omega(s,t)^{\frac{1}{p}}$, then using superadditivity of the control, we have 
$$
\sum_{i}|g _{t_i t_{i+1}}|^p \leq \sum_{i} \omega(t_i, t_{i+1}) \leq \omega(s,t),
$$
for any partition $(t_i)\in  \clP([s,t])$.
Taking supremum over all partitions yields $\omega_g(s,t) \leq \omega(s,t)$. Thus, we could equivalently define a semi-norm  on $C_2^{p-\textnormal{var}}(I;E)$ by 
\begin{equation}\label{equivdefpvar}
|g |_{p-var; [s,t]} = \inf \{ \omega(s,t)^{\frac{1}{p}} : |g_{uv}| \leq \omega(u,v)^{\frac{1}{p}} \textnormal{ for all } (u, v)  \in \Delta_{[s,t]} \} . 
\end{equation}

We shall need a local version of the $p$-variation spaces, for which we restrict the mesh size of the partition  by a control.

\begin{definition}\label{def:variationSpace}
Given an interval $I=[a,b]$, a control $\varpi$ and real number $L> 0$, we denote by $C^{p-\textnormal{var}}_{2, \varpi, L}(I; E)$  the space of continuous two-index maps $g : \Delta_I \rightarrow E$ for which  there exists at least one control $\omega$ such that for every $(s,t)\in \Delta_I$ with  $\varpi(s,t) \leq L$, it holds that
$
|g_{st}|_E \leq \omega(s,t)^{\frac{1}{p}}.
$ We define a semi-norm on this space by
$$
|g |_{p-\textnormal{var}, \varpi,L; I} =\inf \left\{\omega(a,b)^{\frac{1}{p}} :  \omega  \textnormal{ is a control s.t. } |g_{st}| \leq \omega(s,t)^{\frac{1}{p}}, \;\forall (s, t)  \in \Delta_{I} \textnormal{  with } \varpi(s,t) \leq L \right\} . 
$$ 
\end{definition}
\begin{remark}By the above analysis, it is clear that we could equivalently define the semi-norm as
$$
|g |_{p-\textnormal{var}, \varpi,L; I} = \sup_{(t_i)\in \clP_{\varpi,L}(I)}\left(\sum_{i}|g _{t_i t_{i+1}}|^p\right)^{\frac{1}{p}},
$$
where $\mathcal{P}_{\varpi, L }(I)$ denotes the family of all partitions of an interval $I$ such that $\varpi(t_i,t_{i+1}) \leq L$ for all neighboring partition points $t_i$ and $t_{i+1}$. 
It is clear that
\begin{equation} \label{inclusionOfLocalVariation}
C^{p-\textnormal{var}}_{2, \varpi_1, L_1}(I; E) \subset C^{p-\textnormal{var}}_{2, \varpi_2, L_2}(I; E)
\end{equation}
for $\varpi_1 \leq \varpi_2$ and $L_2 \leq L_1$.
\end{remark}
\begin{remark}\label{rem:oneindexpathslocal}
Let $I$ be an interval. We could define the local $p$-variation space for 1-index maps $C^{p-\textnormal{var}}_{\varpi, L}(I; E)$ as above.  However, there is no difference between the local and global spaces; that is, 
$
C^{p-\textnormal{var}}_{\varpi, L}(I; E)= C^{p-\textnormal{var}}(I; E).
$
Indeed, clearly $C^{p-\textnormal{var}}(I; E) \subset C^{p-\textnormal{var}}_{\varpi, L}(I; E)$. To show $C^{p-\textnormal{var}}_{\varpi, L}(I; E) \subset C^{p-\textnormal{var}}(I; E)$,  let $\varpi$ be such there is a partition $(s_j)_{j=1}^J$ of $I$ satisfying $\varpi(s_j, s_{j+1}) \leq L$. Then,  for any partition $(t_i)\in  \clP(I)$, we can always find a refinement $(\tilde{t}_k)$ of $(t_i)$ containing $(s_j)$. It follows from the superadditivity of $\varpi$ that $\varpi(\tilde{t}_k,\tilde{t}_{k+1})\le L$. Moreover, either  an interval $(t_i,t_{i+1})$ does not contain any of the $(s_j)_{j=1}^J$ or it contains  a set $\{s_{j_1(i)}, \ldots,s_{j_{n(i)}(i)} \}$. In the latter case, we have
$$
\delta g_{t_it_{i+1}}=\delta g_{t_is_{j_1(i)}}+\sum_{j=j_1(i)}^{j_{n(i)}(i)-1}   \delta g_{s_j s_{j+1}}+\delta g_{s_{j_{n(i)}(i)}t_{i+1}}.
$$
Thus, for any $g \in C^{p-\textnormal{var}}_{\varpi, L}(I; E)$, we have
$$
\sum_{(t_i)\in \clP(I)} |\delta g_{t_i t_{i+1}}|^p \lesssim_p \sum_{(\tilde{t}_i)\in \clP_{\varpi,L}} |\delta g_{\tilde{t}_i \tilde{t}_{i+1}}|^p\lesssim_p |g |_{p-\textnormal{var}, \varpi,L; I},
$$
and hence $C^{p-\textnormal{var}}_{\varpi, L}(I; E)= C^{p-\textnormal{var}}(I; E)$.
\end{remark}

\medskip

We now introduce the notion of a rough path. For a thorough introduction to the theory of rough paths, we refer the reader to the monographs \cite{MR2314753, FrVi10, FrHa14}. For a two-index map $g: \Delta_{I}\rightarrow \bR$, we define the second order increment operator $$\delta g_{s \theta t} = g_{st} - g_{\theta t} - g_{s \theta}, \quad \forall (s,\theta,t)\in\Delta^{(2)}_I.$$

\begin{definition}\label{defi-rough-path}
Let $K \in\bN $ and $p\in[2,3)$. A continuous $p$-rough path is a  pair 
\begin{equation}\label{p-var-rp}
\bZ=(Z, \mathbb{Z}) \in C^{p-\textnormal{var}}_2 ([0,T];\bR^{K}) \times C^{\frac{p}{2}-\textnormal{var}}_2 ([0,T]; \bR^{K\times K}) 
\end{equation}
that satisfies the Chen's relation 
\begin{equation*}\label{chen-rela}
\delta \mathbb{Z}_{s\theta t}=Z_{s\theta} \otimes Z_{\theta t},  \quad \forall(s, \theta, t) \in \Delta^{(2)}_{[0,T]}.
\end{equation*}
A rough path $\mathbf{Z}=(Z, \mathbb{Z})$ is said to be geometric if it can be obtained as the limit in the   product topology $C^{p-\textnormal{var}}_2 ([0,T];\bR^{K}) \times C^{\frac{p}{2}-\textnormal{var}}_2 ([0,T]; \bR^{K\times K})$  of a sequence of rough paths  $\{(Z^{n},\mathbb{Z}^{n})\}_{n=1}^{\infty}$ such that for each $n\in \bN$, 
$$Z^{n}_{st}:= \delta z^{n}_{st} \quad \textnormal{ and } \quad \mathbb{Z}^{n}_{st}:=\int_s^t \delta z^{n}_{s\theta} \otimes \mathrm{d} z^{n}_\theta ,$$
for some smooth path $z^n:[0,T] \to \bR^K$, where the  iterated integral is a Riemann integral. 
We denote by $\mathcal{C}^{p-\textnormal{var}}_g([0,T];\bR^K)$ the set of geometric $p$-rough paths and endow it with the product topology.
\end{definition}
\begin{remark}
For any continuous $p$-rough path $\bZ=(Z,\mathbb{Z})$, it is clear that we  can always find a control $\omega$ such that for all $(s,t)\in \Delta_T$,
$$
|Z_{st}|^p\le \omega(s,t)\quad \textnormal{and} \quad |\mathbb{Z}_{st}|^{\frac{p}{2}}\le\omega(s,t).
$$
With abuse of notation, we write $\omega=\omega_Z$.
This should compared with \eqref{ineq:UBRcontrolestimates} below.
\end{remark}

Throughout this paper, we will only consider  geometric rough paths. An advantage of working with geometric rough paths is that a first-order chain rule  similar to the one known for smooth paths holds true. We recall that such a chain rule is not
true in It\^o integration theory, in which only a (second order) It\^o formula is available. However, for the Stratonovich  integral, a first order chain rule holds true. Thus, in case of
a Brownian motion, a Stratonovich integral should be used for the construction of the iterated integral if one wishes to lift it to a geometric rough path.

\subsection{Unbounded rough drivers}\label{sec:unboundedrough}

Since the rough perturbation  in \eqref{eq:classicalForm} is  (unbounded) operator valued, it is necessary to generalize the notion of a rough path accordingly. To this end, we define  unbounded rough drivers, which can be regarded as operator valued rough paths with values in a suitable space of unbounded operators.
In what follows, we call a scale any family $(E^{\alpha}, | \cdot |_{\alpha})_{ \alpha \in \bR_+}$ of Banach spaces such that $E^{\alpha+ \beta}$ is continuously embedded into $E^{\alpha}$ for $\beta \in \bR_+$. For $\alpha \in\bR_+,$ we denote by $E^{-\alpha}$ the topological dual of $E^{\alpha}$, and note that, in general, $E^{-0}\neq E^0.$

\begin{definition}
\label{def:urd}
Let $p\in [2,3)$ and $T>0$ be given. A continuous unbounded $p$-rough driver with respect to the scale $(E^{\alpha}, |\cdot |_{\alpha})_{\alpha \in\bR_+}$, is a pair $\mathbf{A} = (A^1,A^2)$ of $2$-index maps such that
 there exists a continuous control $\omega_A$ on $[0,T]$ such that for every $(s,t)\in \Delta_T$,
\begin{equation}\label{ineq:UBRcontrolestimates}
| A^1_{st}|_{\mathcal{L}(E^{-\alpha},E^{-(\alpha+1)})}^p \leq\omega_{A}(s,t) \ \  \text{for}\ \ \alpha \in [0,2], \quad
|A^2_{st}|_{\mathcal{L}(E^{-\alpha},E^{-(\alpha+2)})}^{\frac{p}{2}} \leq\omega_{A}(s,t) \ \ \text{for}\ \ \alpha \in [0,1],
\end{equation}
and   Chen's relation holds true,
\begin{equation}\label{eq:chen-relation}
\delta A^1_{s\theta t}=0,\qquad\delta A^2_{s\theta t}= A^1_{\theta t}A^1_{s\theta},\;\;\forall (s,\theta,t)\in\Delta^{(2)}_T.
\end{equation}
\end{definition}
We will show below that Definition~\ref{def:urd} allows for a formulation of \eqref{eq:classicalForm}, \eqref{eq:aa3} (see Definition \ref{def:solution} and \ref{def:solution2}).

\subsection{Formulation of the equation}
\label{ss:form}

In this section, we derive a rough path formulation of \eqref{eq:classicalForm}, \eqref{eq:aa3}, which will be satisfied by solutions constructed by our main result below, Theorem \ref{existenceThmNoGalerkin}. The main ideas of this step were already discussed in Section \ref{s:intro} in the simpler setting of the transport equation \eqref{eq:puretransport}.

We fix an arbitrary terminal time $T>0$ and viscosity $\nu>0$. Let $d\in \{2,3\}$.  Let $z\in C^{p-\text{var}}([0,T];\bR^K)$ be such that it can be lifted to a  continuous geometric  $p$-rough path $\bZ=(Z,\mathbb{Z})\in \mathcal{C}^{p-\textnormal{var}}_g([0,T];\bR^K)$ for some $p\in [2,3)$. For each $k\in \{1,\ldots,K\}$, assume that  $\sigma_k : \bT^d \rightarrow \bR^d$ is  twice differentiable and divergence-free. Moreover, assume that for all $k\in \{1,\ldots,K\}$,  $\sigma_k$ and its derivatives up to order two are bounded uniformly.
For  given initial condition $u_0\in \bH^0$, we consider the  system of Navier-Stokes equations on $(t,x) \in [0,T] \times \bT^d$ given by
\begin{equation}
\begin{aligned} \label{NSDiffForm}
\partial_t u + (u \cdot \nabla) u  +\nabla p    & =   \nu \Delta u+(\sigma_k\cdot \nabla)u\,\dot{z}^k_t ,  \\
\nabla\cdot u & = 0,   \\
u(0) & = u_0,
\end{aligned}
\end{equation}
where the unknown are the velocity field $u: [0,T] \times \bT^d \rightarrow \bR^d$ and  pressure $p: [0,T] \times \bT^d \rightarrow \bR$.
Here and below, we use the notation
$$
(u \cdot \nabla) u= \sum_{j=1}^d u^j\frac{ \partial u}{ \partial x_j}\quad  \textnormal{ and }\quad
(\sigma_k\cdot \nabla)u\,\dot{z}^k_t=\sum_{k=1}^K(\sigma_k\cdot \nabla)u\,\dot{z}^k_t=\sum_{k=1}^K\sum_{j=1}^d \sigma^j_k\frac{ \partial u}{ \partial x_j}.
$$

The classical way of studying the Navier-Stokes equation in the variational framework is to decouple the velocity field and the pressure into two equations using the Leray projection $P$ defined in Section \ref{ss:notation}. Applying the solenoidal $ P :\W^{\alpha,2} \rightarrow \bH^\alpha$ and gradient projection $Q: \W^{\alpha,2} \rightarrow \bH^{\alpha}_{\perp}$ separately to \eqref{NSDiffForm} yields
\begin{equation} \label{NSDiffFormSystem}
\begin{aligned}
\partial_t u + P[ (u\cdot \nabla) u]&= \nu \Delta u+P[ (\sigma_k\cdot \nabla) u] \dot{z}_t^k,  \\
\nabla p+Q[ (u \cdot \nabla)  u] &=Q [(\sigma_k \cdot \nabla)  u] \dot{z}_t^k.
\end{aligned}
\end{equation}
We let $$\pi: = \int_0^{\cdot} \nabla p_r\,dr.$$ As we did for the pure transport equation  \eqref{eq:puretransport} in the introduction, we integrate the  \eqref{NSDiffFormSystem} over $[s,t]$ and then iterate the equation into itself to  obtain
\begin{equation} \label{NSRoughFormSystem}
\begin{aligned}
\delta u_{st} +\int_s^t P[ (u_r \cdot \nabla)  u_r]\,dr&=\int_s^t \nu\Delta u_rdr+ [A_{st}^{P,1}+A_{st}^{P,2}]u_s +u_{st}^{P, \natural} , \\
\delta \pi_{st}+\int_s^t Q[ (u_r \cdot \nabla ) u_r)]\,dr&= [A_{st}^{Q,1}+A_{st}^{Q,2}] u_s + u_{st}^{Q, \natural}  ,
\end{aligned}
\end{equation}
where 
\begin{gather*}
A^{P,1}_{st}\varphi :=   P  [(\sigma_k \cdot \nabla)  \varphi]   \, Z_{st}^k, \qquad
A^{P,2}_{st}\varphi:=P[(\sigma_k\cdot\nabla) P [(\sigma_i\cdot\nabla)\varphi]]\mathbb{Z}^{i,k}_{st},\\
A^{Q,1}_{st}\varphi :=    Q  [(\sigma_k \cdot \nabla)  \varphi] \, Z_{st}^{k}, \qquad
A^{Q,2}_{st}\varphi :=   Q[(\sigma_k\cdot\nabla) P [(\sigma_i\cdot\nabla)\varphi]] \, \mathbb{Z}^{i,k}_{st}.
\end{gather*}
To do this derivation, let us assume  we have a solution $u \in L_T^2 \bH^1 \cap L_T^{\infty} \bH^0$.
If we set 
$$\mu_{\cdot}=\int_0^{\cdot} \left[\nu\Delta u_r - (u_r \cdot \nabla) u_r\right]  dr,$$ then by \eqref{trilinear form estimate} with $\alpha_1 = \alpha_3 = 1$ and $\alpha_2 = 0$, we have $\mu \in C^{1 -\textnormal{var}}([0,T]; \W^{-1,2})$. Iterating the first equation of \eqref{NSDiffFormSystem}  into itself gives 
\begin{align}
\delta u_{st} & = P \delta \mu_{st} +  \int_s^t P (\sigma_k \cdot  \nabla)   \left(u_s  + P\delta\mu_{sr}+   \int_s^r P (\sigma_i \cdot  \nabla) u_{r_1} \, dz^i_{r_1}\right) \, dz_r^k \notag \\
& = P \delta \mu_{st} +P ( \sigma_k \cdot  \nabla )   u_s Z^k_{st}  + \int_s^t P (\sigma_k \cdot \nabla)   \delta \mu_{sr}  \, dz^k_r +  \int_s^t P (\sigma_k \cdot \nabla) \int_s^r P (\sigma_i \cdot \nabla)  u_{r_1} \, dz^i_{r_1} \, dz^k_r \notag\\
& = P \delta \mu_{st} +P  (\sigma \cdot \nabla_k )  u_s Z^k_{st}  + \int_s^t P (\sigma_k \cdot \nabla )  \delta \mu_{sr}  \, dz^k_r + \notag \\
& \quad +  \int_s^t P (\sigma_k \cdot \nabla )\int_s^r P (\sigma_i \cdot \nabla)  \left(u_s + P \delta \mu_{sr_1}  + P \int_s^{r_1} (\sigma_j \cdot \nabla) u_{r_2} \, dz^j_{r_2}   \right) \, dz^i_{r_1} \, dz^k_r\notag \\
& = P \delta \mu_{st} +P ( \sigma_k \cdot \nabla )  u_s Z^k_{st}  + P( \sigma_k \cdot \nabla)  P (\sigma_i \cdot \nabla)   u_s \mathbb{Z}^{i,k}_{st} + \int_s^t P (\sigma_k \cdot \nabla )  \delta \mu_{sr}  \, dz^k_r + \notag \\
& \quad +  \int_s^t P( \sigma_k \cdot \nabla )\int_s^r P (\sigma_i \cdot \nabla)  \left( P \delta \mu_{sr_1}  + P \int_s^{r_1}( \sigma_j \cdot \nabla) u_{r_2} \, dz^j_{r_2}  \right)  \, dz^i_{r_1} \, dz_r^k\notag \\
& = P \delta \mu_{st} +P [ (\sigma_k \cdot \nabla )  u_s ]Z^k_{st}  + P [(\sigma_k  \cdot \nabla ) P [(\sigma_i \cdot \nabla )  u_s]] \mathbb{Z}^{i,k}_{st} + u_{st}^{P, \natural},\label{eq:iterationP}
\end{align}
where 
\begin{align*}
u_{st}^{P, \natural} &:= \int_s^t P  (\sigma_k \cdot \nabla  ) P\delta \mu_{sr}  \, dz^k_r \\
&\quad +  \int_s^t P (\sigma_k \cdot \nabla )\int_s^r P ( \sigma_i \cdot \nabla ) \left( P \delta \mu_{sr_1}  + P \int_s^{r_1} (\sigma_j \cdot \nabla) u_{r_2} \, dz^j_{r_2}   \right) \, dz^i_{r_1} \, dz^k_r.
\end{align*}
We expect $u_{st}^{P, \natural}$  be in $C^{\frac{p}{3} -\textnormal{var}}_2([0,T]; \bH^{-3})$ since $\mu \in C^{1 -\textnormal{ var}}([0,T]; \W^{-1,2})$ and $u \in L_T^{\infty} \bH^0$. 

Note  that  $Q \mu = - \int_0^{\cdot}Q [ (u_r \cdot \nabla ) u_r] dr.$ Then, iterating the first equation of \eqref{NSDiffFormSystem}   into second equation, we find
\begin{align*}
\delta \pi_{st} &  =  Q \delta \mu_{st} + Q \int_s^t ( \sigma_k \cdot \nabla ) u_r \, dz^k_r  \\
& =   Q \delta \mu_{st} +  Q \int_s^t ( \sigma_k \cdot \nabla ) \left( u_s + P \delta \mu_{sr}  + P \int_s^r ( \sigma_i \cdot \nabla ) u_{r_1} \, dz^i_{r_1} \right)  \, dz^k_r \\
& = Q \delta \mu_{st}+ Q  ( \sigma_k \cdot \nabla ) u_s Z^k_{st} +     Q \int_s^t ( \sigma_k \cdot \nabla )  P \delta \mu_{sr}    \, dz^k_r  \\
& \quad + Q \int_s^t ( \sigma_k \cdot \nabla ) P \int_s^r ( \sigma_i \cdot \nabla ) \left( u_{s}  +  P \delta \mu_{sr_1}  + P \int_s^{r_1} ( \sigma_j \cdot \nabla ) u_{r_2} \, dz^j_{r_2} \,  \right) dz^i_{r_1} \, dz^k_r \\
& =  Q \delta \mu_{st} + Q  [( \sigma_k \cdot \nabla ) u_s] Z^k_{st} + Q  [( \sigma_k \cdot \nabla ) P  [( \sigma_i \cdot \nabla )  u_{s}]] \mathbb{Z}^{i,k}_{st}  + u_{st}^{Q, \natural},
\end{align*}
where 
\begin{align*}
u_{st}^{Q, \natural} &= Q \int_s^t ( \sigma_k \cdot \nabla )  P \delta \mu_{sr}    \, dz_r^k\\
&\quad   + Q \int_s^t ( \sigma_k \cdot \nabla ) P \int_s^r ( \sigma_i \cdot  \nabla ) \left(   P \delta \mu_{sr_1}  + P \int_s^{r_1} ( \sigma_j \cdot \nabla ) u_{r_2} \, dz_{r_2}^j) \, dz_{r_1}^i \right)  \, dz_r^k,
\end{align*}
which is  expected to be in $C^{\frac{p}{3} - \textnormal{var}}_2([0,T]; \bH_{\perp}^{-3})$.

Equations \eqref{NSRoughFormSystem} are to be understood in the sense that we \emph{define} the remainder terms $u^{P, \natural}$ and $u^{Q, \natural}$ from the solution $u$ and $\pi$, and have to verify that they are indeed negligible remainders, namely, they are of order $o(|t-s|)$. This will be made precise in Definition \ref{def:solution2} below.

The pair $\mathbf{A}^P=(A^{P,1}, A^{P,2})$ is an unbounded rough driver (Definition \ref{def:urd}) on the scale $(\bH^{\alpha})_{\alpha \in \bR_+}$. Indeed,  the existence of a control  $\omega_{\mathbf{A}^P}$ such that \eqref{ineq:UBRcontrolestimates} holds follows from the discussion in Section \ref{ss:notation} and the fact that $(Z,\mathbb{Z})$ is a  $p$-rough path (Definition \ref{defi-rough-path}), which also implies  Chen's relation \eqref{eq:chen-relation}.  We note that control  $\omega_{\mathbf{A}^P}$  can be chosen to satisfy 
\begin{equation}\label{ineq:Leraycontrolestimate}
\omega_{A^P}(s,t) \le C\omega_{Z}(s,t),\;\;\forall (s,t)\in \Delta_T,
\end{equation}
for a constant $C>0$  depending only on $d$ and the bounds on $\sigma=(\sigma_1,\ldots,\sigma_K)$ and its derivatives up to order two.

The  pair $\mathbf{A}^Q=(A^{Q,1}, A^{Q,2})$ satisfies \eqref{ineq:UBRcontrolestimates} for the scale  $(\bH^{\alpha}_{\perp})_{\alpha\in \bR_+}$ with a control $\omega_{A^Q}$, which also satisfies the bound  \eqref{ineq:Leraycontrolestimate}.  However, $\mathbf{A}^Q$ is not  an unbounded rough driver  since it fails to satisfy Chen's relation \eqref{eq:chen-relation}. Nevertheless, it satisfies
\begin{equation} \label{quasiChen}
\delta A_{s \theta t}^{Q,2} = A_{\theta t}^{Q,1} A_{s \theta}^{P,1},\quad\text{for all}\quad  (s,\theta,t)\in\Delta^{(2)}_T,
\end{equation}
which  is the correct Chen's relation for the system of equations \eqref{NSDiffFormSystem} needed top establish the required time regularity of the remainder $u^{Q,\natural}$ (see Section \ref{sec: a priori} and Lemma \ref{PressureRemainderLemma} ). 

We will now define our first notion of solution to \eqref{NSDiffForm}.

\begin{definition}\label{def:solution2}
A pair of weakly continuous functions $(u, \pi) : [0,T] \rightarrow \bH^0 \times \bH_{\perp}^{-3}$ is called a solution of \eqref{NSDiffForm} if $u\in L^2_T\bH^1\cap L^{\infty}_T\bH^0$ and   $u^{P,\natural} : \Delta_T \rightarrow \bH^{-3}$ and $u^{Q, \natural}: \Delta_T \rightarrow \bH_{\perp}^{-3} $ defined for all $ \phi\in \bH^{3}$, $\psi \in \bH_{\perp}^3$ and $(s,t)\in \Delta_T$ by 
\begin{align} 
u_{st}^{P,\natural}(\phi)& :=   \delta u_{st} (\phi ) +  \int_s^t \left[\nu (\nabla u_r, \nabla   \phi) + B_P(u_r)(\phi) \right]\,dr   -  u_s([A_{st}^{P,1,*} +A_{st}^{P,2,*}]\phi) , \label{SystemSolutionU} \\
u_{st}^{Q,\natural}(\psi) & :=    \delta \pi_{st} (\psi ) + \int_s^t  B_Q(u_r)(\psi)\,dr   -  u_s ( [A_{st}^{Q,1,*}+A_{st}^{Q,2,*}] \psi ) , \label{SystemSolutionPi}
\end{align}
satisfy
\begin{equation} \label{SystemSolutionRemainder}
u^{P,\natural} \in C^{\frac{p}{3}-\textnormal{var}}_{2, \varpi,L}([0,T]; \bH^{-3}) \qquad \textnormal{and} \qquad u^{Q,\natural} \in C^{\frac{p}{3}-\textnormal{var}}_{2, \varpi,L}([0,T]; \bH_{\perp}^{-3}),
\end{equation}
for some control $\varpi$ and $L> 0$.
\end{definition}
\begin{remark}
Applying  \eqref{trilinear form estimate} with $\alpha_1=0, \alpha_2=2,$ and $\alpha_3=0$, we get
$$
B_P(u)(\phi)=B_P(u,u)(\phi)=B_P(u,\phi)(u)\lesssim |u|_{0}^2|\phi|_3,
$$
from which it follows that the $dr$-integral in  \eqref{SystemSolutionU} is well-defined since $u\in \bL^{\infty}_T\bH^0$.  One could  also obtain an estimate that requires less regularity on $\phi$ by  applying     \eqref{trilinear form estimate} with  $\alpha_1=1, \alpha_2=0, $ and $\alpha_3=1$ to get,
$$
|B_P(u)(\phi)|\lesssim |u|_1^2|\phi|_1,
$$
from which it follows that the $dr$-integral in   \eqref{SystemSolutionU}  is well-defined since $u\in \bL^{2}_T\bH^1$.  However,  we must test by $\phi \in \bH^3$ to ensure that the remainder term $u^{P,\natural}(\phi)$ has the required time regularity.  An analogous argument holds for the $B_Q$ term in \eqref{SystemSolutionRemainder}.
\end{remark}
\begin{remark}
In \eqref{SystemSolutionU} and \eqref{SystemSolutionRemainder}, we opt for distributional evaluation notation for most terms, and  continue to do so throughout the paper. That is,
$$
u_{st}^{P,\natural}(\phi)=\langle u_{st}^{P,\natural},\phi\rangle_{-3,3}, \;\; \delta u_{st}(\phi)=(\delta u_{st},\phi)_0,\quad u_s([A_{st}^{P,1,*} +A_{st}^{P,2,*}]\phi)_0= (u_s, [A_{st}^{P,1,*} +A_{st}^{P,2,*}]\phi)_0,
$$
$$
u_{st}^{Q,\natural}(\phi)=\langle u_{st}^{Q,\natural},\psi \rangle_{-3,3}, \quad u_s ( [A_{st}^{Q,1,*}+A_{st}^{Q,2,*}] \psi)_0= (u_s, [A_{st}^{Q,1,*} +A_{st}^{Q,2,*}]\psi)_0.
$$
\end{remark}
\begin{remark}
Due to \eqref{inclusionOfLocalVariation}, there is no restriction in taking the same $\varpi$ and $L>0$ for both local variation spaces in \eqref{SystemSolutionRemainder}.
\end{remark}

We will now discuss an alternative  way of formulating the equation. We can arrive at this formulation  by performing an iteration directly on \eqref{NSDiffForm}:
\begin{align*}
\delta u_{st} 
& = \delta \mu_{st} - \delta \pi_{st} + ( \sigma_k \cdot \nabla ) u_s  Z^k_{st} +  ( \sigma_k \cdot \nabla ) ( \sigma_i \cdot \nabla )  u_s  \mathbb{Z}^{i,k}_{st}  \\
& \quad+ \int_s^t ( \sigma_k \cdot \nabla )  \delta \mu_{sr}  \, dz^k_r   - \int_s^t ( \sigma_k \cdot \nabla ) \delta \pi_{sr}  \, dz_r^k\\
&\quad + \int_s^t ( \sigma_k \cdot \nabla ) \int_s^r ( \sigma_i \cdot \nabla ) \left(\delta \mu_{s r_1} - \delta \pi_{s r_1} + \int_s^{r_1} ( \sigma_j \cdot \nabla ) u_{r_2}  \, dz^j_{r_2} \right)  \, dz^i_{r_1}  \, dz^k_r . 
\end{align*}
The integral $\int_s^t  ( \sigma_k \cdot \nabla ) \delta \pi_{sr}   \,  dz^k_r $ is not regular enough in time for it to be a negligible remainder. Indeed, we expect $\pi$ to have finite $p$-variation, so that $\int_0^{\cdot} ( \sigma_k \cdot \nabla ) \delta \pi_{sr}  \,  dz^k_r $ should only have finite $\frac{p}{2}$-variation.  
If we define 
$$
\bar{u}^{\natural}_{st} = \int_s^t ( \sigma_k \cdot \nabla )   \delta \mu_{sr}   \, dz^k_r + \int_s^t ( \sigma_k \cdot \nabla )  \int_s^r ( \sigma_i \cdot \nabla ) \left(\delta \mu_{s r_2} - \delta \pi_{s r_2} + \int_s^{r_1}  ( \sigma_j \cdot \nabla ) u_{r_2}  \, dz^j_{r_2} \right)  \, dz^i_{r_1}  \, dz^k_r,
$$
then   we expect $\bar{u}^{\natural}$ to be in  $C_2^{\frac{p}{3} -\textnormal{var}}([0,T]; \W^{-3,2})$. Moreover, we have
$$
\delta u_{st} = \delta \mu_{st} - \delta \pi_{st} + ( \sigma_k \cdot \nabla )  u_s  Z^k_{st} +  ( \sigma_k \cdot   \nabla )  ( \sigma_i \cdot \nabla )  u_s  \mathbb{Z}^{i,k}_{st}  + \bar{u}_{st}^{\natural}  -\int_s^t ( \sigma_k \cdot \nabla )  \delta \pi_{sr} \, dz^k_r.
$$
In order to complete the formulation, we use the equation \eqref{NSDiffFormSystem}  for $\pi$ to deduce 
\begin{align*}
- \int_s^t ( \sigma_k \cdot \nabla )  \delta \pi_{sr}dz^k_r %&  =    \int_s^t ( \sigma_k \cdot \nabla )  Q \delta \mu_{sr}  dz^k_r - \int_s^t ( \sigma_k \cdot \nabla )  \int_s^r Q ( \sigma_i \cdot  \nabla ) u_{r_1}  dz^i_{r_1}  dz^k_r \\
%& =  \int_s^t ( \sigma_k \cdot \nabla )  Q \delta  \mu_{sr}  dz^k_r  \\
%&\quad - \int_s^t ( \sigma_k \cdot \nabla )  \int_s^r Q ( \sigma_i \cdot  \nabla )  \left( u_s  + \delta \mu_{sr_1} - \delta \pi_{sr_1} + \int_s^{r_1} ( \sigma_j \nabla ) u_{r_2}  dz^j_{r_2} \right) dz^i_{r_1}  dz^k_r \\
& =   - ( \sigma_k \cdot \nabla )   Q  (( \sigma_i \cdot \nabla )  u_s )  \mathbb{Z}^{i,k}_{st} + \int_s^t ( \sigma_k \cdot \nabla )  Q \delta  \mu_{sr}  dz^k_r  \\
& \quad - \int_s^t ( \sigma_k \cdot \nabla )  \int_s^r Q  ( \sigma_i \cdot \nabla )  \left( \delta \mu_{sr_1} - \delta \pi_{sr_1} + \int_s^{r_1} ( \sigma_j \cdot \nabla ) u_{r_1} dz_{r_1}^j \right)  dz_{r_1}^i  dz^k_r .
\end{align*}
All the terms above except for $( \sigma_k \cdot \nabla )   Q  \left[( \sigma_i \cdot \nabla )  u_s \right]  \mathbb{Z}^{i,k}_{st}$ belong to $C_2^{\frac{p}{3} -\textnormal{var}} ([0,T]; \W^{-3,2})$, and hence we may include them in a new remainder 
\begin{align*}
u^{\natural}_{st} &:= \bar{u}^{\natural}_{st} - \int_s^t ( \sigma_k \cdot \nabla )  Q \delta  \mu_{sr}  \, dz^k_r \\
&\quad  \quad  - \int_s^t ( \sigma_k \cdot \nabla )  \int_s^r Q  ( \sigma_i \cdot \nabla )  \left( \delta \mu_{sr_1} + \delta \pi_{sr_1}+ \int_s^{r_1} ( \sigma_j \cdot \nabla ) u_{r_1} \, dz^j_{r_2} \right)  \, dz^i_{r_1}  \, dz^k_r .
\end{align*}
Combining the above, we obtain
\begin{align*}
\delta u_{st} & = \delta \mu_{st} - \delta \pi_{st} + ( \sigma_k \cdot \nabla )  u_s  Z^k_{st} +  ( \sigma_k \cdot \nabla )  ( \sigma_i \cdot  \nabla )  u_s \mathbb{Z}^{i,k}_{st}   -  ( \sigma_k \cdot \nabla )  Q \left[ \sigma_i \cdot \nabla   u_s\right]  \mathbb{Z}^{i,k}_{st}   + u_{st}^{\natural} \\
& = \delta \mu_{st} - \delta \pi_{st} + ( \sigma_k \cdot \nabla )  u_s  Z^k_{st} +  ( \sigma_k \cdot  \nabla )  P  \left[( \sigma_i \cdot \nabla )  u_s\right]   \mathbb{Z}^{i,k}_{st}    + u_{st}^{\natural}.
\end{align*}
Thus, the pair $\mathbf{A}=(A^1,A^2)$  defined by $$A_{st}^1 \varphi = ( \sigma_k \cdot \nabla )  \varphi  Z^k_{st},\quad  \qquad A_{st}^2 \varphi = ( \sigma_k \cdot \nabla )  P  \left[( \sigma_k \cdot \nabla ) \varphi \right]  \mathbb{Z}^{i,k}_{st}$$
satisfies  \eqref{ineq:UBRcontrolestimates} for the scale  $(\bW^{\alpha,2})_{\alpha\in \bR_+}$ with  control $\omega_{A}$.
However, this pair does not satisfy Chen's relation \eqref{eq:chen-relation}, but does satisfy
\begin{equation}\label{quasiChen2}
\delta A^2_{s \theta t} = A_{\theta t}^1 P A_{s \theta}^1 ,\quad \forall (s,\theta,t)\in\Delta^{(2)}_T.
\end{equation}
Since $\mathbf{A}^P=P\mathbf{A}$ and $\mathbf{A}^Q=Q\mathbf{A}$,  the controls $\omega_{A^P}, \omega_{A^Q},$ and $\omega_{A}$ can be chosen so that
$$
\omega_{A^P}(s,t),\omega_{A^Q}(s,t)\le \omega_{A}(s,t)\le C\omega_{Z}(s,t),\;\;\forall (s,t)\in\Delta_T,
$$
where $C$ is a constant depending only on $d$ and the bounds on $\sigma=(\sigma_1,\ldots,\sigma_K)$ and its derivatives up to order two.

Thus, alternatively, we may  formulate a solution to \eqref{NSDiffForm} as follows.

\begin{definition}\label{def:solution}
A pair of weakly continuous functions $(u, \pi) : [0,T] \rightarrow \bH^0 \times \bH_{\perp}^{-3}$ is called a solution of \eqref{NSDiffForm} if  $u\in L^2_T\bH^1\cap L^{\infty}_T\bH^0$ and  $u^{\natural} : \Delta_T \rightarrow \W^{-3,2}$ defined  for all $ \phi\in \bW^{3,2}$ and $(s,t)\in \Delta_T$ by  
\begin{align}
u_{st}^{\natural}(\phi)&=  \delta u_{st} (\phi ) +  \int_s^t \left[\nu(\nabla u_r, \nabla   \phi) + B(u_r)(\phi)\right]\,dr   -  u_s ([A_{st}^{1,*} +A_{st}^{2,*}]\phi ) + \delta \pi_{st} (\phi) ,\label{RNSweak}
\end{align}
satisfies $u^{\natural} \in C^{\frac{p}{3}-\textnormal{var}}_{2, \varpi,L}([0,T]; \W^{-3,2})$ for some control $\varpi$ and $L> 0$. 
\end{definition}

The following lemma says that both formulations were derived in a consistent way and are equivalent.

\begin{lemma}\label{lem:equiv}
Definition \ref{def:solution2} and  \ref{def:solution} are equivalent.
\end{lemma}
\begin{proof}
Clearly, $P A_{st}^{i} = A_{st}^{P,i}$ and $Q A_{st}^{i} = A_{st}^{Q,i}$ for $i\in \{1,2\}$. Moreover, the mapping 
\begin{align*}
C^{\frac{p}{3} -\textnormal{var}}_{2, \varpi,L}([0,T]; \W^{-3,2})  & \rightarrow  C^{\frac{p}{3} -\textnormal{var}}_{2,\varpi,L}([0,T]; \bH^{-3}) \times C^{\frac{p}{3} -\textnormal{var}}_{2,\varpi,L}([0,T]; \bH^{-3}_{\perp}) \\
u^{\natural} & \mapsto (u^{P,\natural},u^{Q,\natural}):= (Pu^{\natural}, Qu^{\natural}) 
\end{align*}
is continuous and invertible with inverse 
\begin{align*}
C^{\frac{p}{3} -\textnormal{var}}_{2,\varpi_1,L_1}([0,T]; \bH^{-3}) \times C^{\frac{p}{3} -\textnormal{var}}_{2,\varpi_2,L_2}([0,T]; \bH^{-3}_{\perp})   & \rightarrow  C^{\frac{p}{3} -\textnormal{var}}_{2, \varpi,L}([0,T]; \W^{-3,2}) \\
(u^{P,\natural}, u^{Q,\natural}) & \mapsto u^{P,\natural}+ u^{Q,\natural}
\end{align*}
where $\varpi := \varpi_1 + \varpi_2$ and $L := L_1 \wedge L_2$. The rest of the proof is straightforward.
\end{proof}
In  the remainder of the paper, we use Definition \ref{def:solution2}.

\subsection{Main results}

Let us now formulate our main results.

\begin{theorem} \label{existenceThmNoGalerkin}
Let $d\in \{2,3\}.$  Assume that for each $k\in \{1,\ldots,K\}$,  $\sigma_k : \bT^d \rightarrow \bR^d$  and its derivatives up to order two are bounded uniformly and that $\sigma_k$ is divergence-free. For a given $u_0\in \bH^0$ and $\mathbf{Z}\in  \mathcal{C}^{p-\textnormal{var}}_g([0,T];\bR^K)$,  there exists a solution of \eqref{NSDiffForm} in the sense of Definition \ref{def:solution2} satisfying the energy inequality
$$
\sup_{ t \in [0,T]} |u_t|^2_0 + \int_0^T |\nabla u_r|_0^2\,dr \le |u_0|_0^2.
$$ 
Moreover, $u\in C^{p-\textnormal{var}}([0,T];\bH^{-1})$ and $\pi \in C^{p-\textnormal{var}}([0,T];\bH^{-3}_{\perp})$.
\end{theorem}
The proof of this result is presented in Section \ref{Section:Galerkin} as a consequence of the stronger statement in Theorem \ref{existenceThm}. It proceeds in two steps: first (see Section \ref{s:galerkin}), the velocity field is constructed using compactness as a limit of suitable Galerkin approximations combined with an approximation of the driving signal $z$ by smooth paths in , second, the pressure is recovered (see Section \ref{s:pressure}).

For two space dimensions and constant vector fields,  we prove that the solution $(u,\pi)$ is unique  as a consequence of the stronger statement  Theorem \ref{contractiveTheorem}. In Section \ref{s:uniq}, we prove uniqueness via a tensorization argument, which allows us to estimate the difference of two solutions by the difference of their initial conditions. We remark that one cannot directly use  the techniques from \cite{DeGuHoTi16}, since this way of approximating the Dirac-delta violates the divergence-free condition.  

\begin{theorem} \label{contractiveTheoremNoContraction}
If $d=2$ and $\sigma_k$ is constant function of $x\in \bT^d$ for all $k\in \{1,\dots, K\}$, then for a given $u_0\in \bH^0$ and $\mathbf{Z}\in  \mathcal{C}^{p-\textnormal{var}}_g([0,T];\bR^K),$ there exists a unique solution of \eqref{NSDiffForm}. Moreover, $u\in C_T\bH^0$, $\pi \in C^{p-\textnormal{var}}([0,T];\bH^{-1}_{\perp})$, and
$$
\sup_{ t \in [0,T]} |u_t|^2_0 +2\nu \int_0^T |\nabla u_r|_0^2\,dr= |u_0|_0^2.
$$
\end{theorem}

Owing to Theorem \ref{contractiveTheoremNoContraction}, in dimension two, there exists a solution map $\Gamma$  that maps  every   initial condition $u_0\in \bH^0$,  family of constant vector fields  $\sigma_k$, $k\in \{1,\ldots,K\}$, and  continuous geometric $p$-rough path $\bZ=(Z,\mathbb{Z})$
to a unique solution $(u,\pi)$  of  \eqref{NSDiffForm}.
The following stability result is proved in Section \ref{s:stab}. 

\begin{corollary}\label{cor:stability}
In dimension two and for  constant vector fields $\sigma_k$, $k\in \{1,\ldots,K\}$,  the solution map 
\begin{align*}
\Gamma:\bH^0\times\bR^{2\times K}\times \mathcal{C}^{p-\textnormal{var}}_g([0,T];\bR^K)&\to L^2_T\bH^0\cap C_T\bH_w^0\times C^{1-\textnormal{var}}([0,T]; \bH_{\perp}^{-2})\\
(u_0,\sigma,\bZ)&\mapsto (u,\pi)
\end{align*}
is continuous.
\end{corollary}

\begin{remark}
It is tempting to try and rewrite \eqref{eq:classicalForm} using a flow-transformation by following the ideas in\cite{FrOb14, CaFrOb11} and  \cite{DiFrOb14}.  More specifically, suppose that there is sufficiently regular invertible map $\varphi : [0,T] \times \bT^d \rightarrow \bT^d $ such that 
$$
\dot{\varphi}_t(x) =   \dot{a}_t( \varphi_t(x))  , \hspace{.3cm} \varphi_0(x) = 0 ,
$$
and let us define $\varv_t(x) := u_t( \varphi_t(x))$. Differentiating in time, we find
\begin{align*}
\partial_t \varv_t(x) & = \partial_t u_t(\varphi_t(x)) +  \dot{a}_t( \varphi_t(x))  \cdot \nabla u_t( \varphi_t(x) )  \\
& =  \nu \Delta u_t( \varphi_t(x)) - u_t( \varphi_t(x)) \cdot  \nabla u_t( \varphi_t(x)) - \nabla p _t( \varphi_t(x)) ,
\end{align*}
which could be rewritten  in terms of $v$ using $\nabla \varv_t(x) = \nabla u_t( \varphi_t(x)) \nabla \varphi_t(x)$ provided $\varphi_t(\cdot)$ is a diffeomorphism. 
If we assume all the driving vector fields are divergence-free, then we have $ \textnormal{det} (\nabla \varphi_t(x) ) = 1$ so that the equation for $\varv$ is a Navier-Stokes-type  equation, including coefficients from a unimodular matrix depending on $t$ and $x$. This could account for further difficulties, but it seems plausible that one can solve such an equation. The added value of the construction we present in this paper is that it allows for an intrinsic notion of solution to \eqref{eq:classicalForm} and estimates of the corresponding rough integral.
\end{remark}

\begin{remark}
In three dimensions, it is known that the Stratonovich Navier-Stokes equation
\begin{align*} 
du + (u \cdot \nabla) u\, dt   +  \nabla p& =  \nu \Delta u\, dt   + \nabla u \circ dw
\end{align*}
has a probabilistically weak solution (see, e.g., \cite{brzezniak1992stochastic, flandoli1995martingale,mikulevicius2005global}). Nevertheless,   whether the solution probabilistically strong is still an open question. In other words, it is not known whether the solution to the above equation is adapted to the filtration generated by the Wiener process $w$. Even though a prime example of a driving rough path in our equation is a Wiener process with its Stratonovich lift and solving rough PDEs corresponds to a non-probabilistic (pathwise) construction of solutions, we still can not answer this question at this point. The reader should note that using the compactness criterion Lemma \ref{CompactnessLemma}, we obtain a  subsequence of the approximate solutions that a priori depends the randomness variable $\omega$ (not a control).
The question whether the full sequence converges is very difficult to answer, as it is intimately related to the issue of uniqueness. Indeed, if uniqueness held true in three-dimensions, then every subsequence of $\{ u^N\}_{N=1}^{\infty}$ would converge to the same limit, and hence the full sequence would converge. As a consequence,  the proof of stability in Corollary \ref{cor:stability} would imply that the solution $(u,\pi)$ depends continuously on the given data $(u_0,\sigma,\bZ)$ and is thus adapted to the filtration generated by the Brownian motion.
\end{remark}
\section{A priori estimates of remainders}
\label{sec: a priori}

In this section, we derive a priori estimates of the remainder terms $u^{P,\natural}$ and $u^{Q,\natural}$ and  $|u|_{p-\textnormal{var};[0,T];\bH^{-1}}$.
Let $(u,\pi)$ be a solution of \eqref{NSDiffFormSystem} in the sense of Definition \ref{def:solution2}. For $t \in [0,T]$, we let 
$$
\mu_{t}(\phi) =  - \int_0^t \left[\nu (\nabla u_r,\nabla \phi) +B_P(u_r)(\phi) \right]\,dr, \quad \phi \in \bH^1.
$$ 
It follows that for  $(s,t)\in \Delta_T$,
\begin{equation}\label{RNS2}
\delta u_{st} =  \delta \mu_{st} + A_{st}^{P,1} u_s+ A_{st}^{P,2}u_s + u^{P,\natural}_{st},
\end{equation}
where the equality holds in $\bH^{-3}$.
For all $(s,t) \in \Delta_T$, let
$$
\omega_{\mu}(s,t) =  \int_s^t  (1 + |u_r|_1)^2 \, dr,
$$
where we recall that $|\cdot|_1$ denotes the $\bH^1$-norm.
Since $u \in L_T^2 \bH^1$, $\omega_{\mu}$ is a control. Using  \eqref{trilinear form estimate} with $\alpha_1 = \alpha_3 = 1$ and $\alpha_2 = 0$, we obtain $|B_P(u_r)|_{-1} \lesssim |u_r|_1^2$, and hence
$
|\delta \mu_{st} |_{-1} \lesssim \omega_{\mu}(s,t).
$

We begin with an important lemma which provides an estimate of $u^{P,\natural}$ in terms of the given data. The following result is a special case of \cite[Theorem 2.5]{DeGuHoTi16}, but we include a proof for the readers convenience. Let us  define the map
\begin{equation}\label{eq:defin of sharp}
u^\sharp_{st}:=\delta u_{s t}-A^{P,1}_{st}u_s=\delta\mu_{st}+A^{P,2}_{st}u_s+u^{P,\natural}_{st}.
\end{equation}
The first expression for $u^\sharp_{st}$ consists of terms that are less regular in time  and more regular in space than the second expression for  $u^\sharp_{st}$. We use this fact along with the smoothing operators and the sewing lemma \eqref{ineq:sewing estimate} to estimate the remainder terms. 
\begin{lemma} \label{Thm2.5}
Assume that $(u,\pi)$ solves  \eqref{NSDiffForm} according to Definition \ref{def:solution2}. For $(s,t)\in \Delta_T$ such that  $\varpi(s,t)\le L$, let $\omega_{P, \natural }(s,t) := |u^{P, \natural}|^{\frac{p}{3}}_{\frac{p}{3} -var; [s,t];\bH^{-3}}.$
Then  there is a constant $\tilde{L}>0$, depending only on $p$ and $d$,  such that  for all $(s,t)\in \Delta_T$ with  $\varpi(s,t)\le L$ and $\omega_{A}(s,t)\leq \tilde{L}$,
\begin{equation} \label{RemainderEstimateMu}
\omega_{P, \natural }(s,t)  \lesssim_{p}|u|^{\frac{p}{3}}_{L^{\infty}_T\bH^0}  \omega_A(s,t) + \omega_{\mu}(s,t)^{\frac{p}{3}} (\omega_A(s,t)^{\frac{1}{3}}  + \omega_A(s,t)^{\frac{2}{3}})
\end{equation}
and
\begin{equation} \label{RemainderEstimateWithoutMu}
\omega_{P, \natural }(s,t)  \lesssim_{p} |u|^{\frac{p}{3}}_{L^{\infty}_T\bH^0}  \omega_A(s,t) + ( 1 + |u|_{L^{\infty}_T\bH^0} )^{\frac{2p}{3}}(t-s)^{\frac{p}{3}} \omega_A(s,t)^{\frac{1}{12}}.
\end{equation}
\end{lemma}
\begin{proof}
Recall that the second-order increment operator $\delta$ is defined on two index maps $g:\Delta_T^{(2)}\rightarrow \bR$  by $\delta g_{s \theta t} := g_{st} - g_{\theta t} - g_{s \theta}$ for all $(s,\theta,t)\in \Delta_T^{(2)}$. It is easy to see that for a one-index map $f$, we have $\delta (\delta f)_{s \theta t} = 0$.
Applying $\delta$ to \eqref{SystemSolutionU}, we find  that for all $\phi \in \bH^{3}$ and $(s,\theta,t)\in \Delta_T^{(2)},$
$$
\delta u_{s \theta t}^{P, \natural }( \phi) = \delta u_{s \theta} (A_{\theta t}^{P,2,*} \phi) + u^\sharp_{s\theta}(A^{P,1,*}_{\theta t} \phi),
$$
where $u^\sharp_{s\theta}$ is defined in \eqref{eq:defin of sharp}.
We decompose $\delta u^{P, \natural }_{s \theta t} (\phi)$  into a smooth (in space) and non-smooth part using the smoothing operator $J^{\eta}$ to get $$u^{P, \natural }_{s \theta t} (\phi)  = \delta u^{P, \natural }_{s \theta t} (J^{\eta} \phi) + \delta u^{P, \natural }_{s \theta t} ((I - J^{\eta} )\phi),$$ for some $\eta \in (0,1]$ that will be specified later. We will now proceed  to analyze term-by-term. 
To estimate the non-smooth part, we use \eqref{smoothingOperator} and that $u^{\natural}_{s\theta}= \delta u_{s \theta}-A^{P,1}_{s\theta}u_s$ to obtain
\begin{align*}
\left|\delta u^{P, \natural }_{s \theta t} ((I - J^{\eta} )\phi)\right| & \le    |u|_{L^{\infty}_T\bH^0}\left( \left|A_{\theta t}^{P,1*} ((I-J^{\eta}) \phi)\right|_0  +  \left|A_{s \theta}^{P,1,*} A_{\theta t}^{P,1*}( (I-J^{\eta}) \phi)\right|_0  + \left|A^{2*}_{\theta t} ((I - J^{\eta})\phi)\right|_0\right) \\
& \lesssim |u|_{L^{\infty}_T\bH^0} \left(\omega_A(s,t)^{\frac{1}{p}} |(I-J^{\eta}) \phi|_1 + \omega_A(s,t)^{\frac{2}{p}} |(I-J^{\eta}) \phi|_2 \right)\\
& \lesssim |u|_{L^{\infty}_T\bH^0} \left( \omega_A(s,t)^{\frac{1}{p}} \eta^2 + \omega_A(s,t)^{\frac{2}{p}} \eta \right) |\phi|_3.
\end{align*}
In order to estimate  the smooth part, we use  the form $u^\sharp_{st}=  \delta \mu_{s \theta } + A_{s \theta}^{P,2} u_s + u_{s \theta}^{P, \natural }$ to get
\begin{align*}
\delta u_{s \theta t}^{P, \natural } (J^{\eta} \phi) & =  \delta \mu_{s \theta}  (A^{P,1,*}_{\theta t} J^{\eta} \phi) + u_s( A^{P,2,*}_{s \theta} A_{\theta t}^{P,1,*} J^{\eta} \phi) + u_{s \theta}^{P, \natural }(A_{\theta t}^{P,1,*} J^{\eta} \phi) \\
& \quad+ \delta \mu_{s \theta}  (A^{P,2,*}_{\theta t} J^{\eta} \phi) + u_s( A^{P,1,*}_{s \theta} A_{\theta t}^{P,2,*} J^{\eta} \phi)+ u_s( A^{P,2,*}_{s \theta} A_{\theta t}^{P,2,*} J^{\eta} \phi) + u_{s \theta}^{P, \natural }(A_{\theta t}^{P,2,*} J^{\eta} \phi) ,
\end{align*}
Estimating each term and using  \eqref{smoothingOperator}, for all $(s,\theta,t)\in \Delta_T^{(2)}$ such that $\varpi(s,t)\le L$,  we find
\begin{align}
|\delta u^{P, \natural }_{s \theta t} (J^{\eta} \phi ) | & \lesssim \omega_{\mu}(s,t) \omega_A(s,t)^{\frac{1}{p}} |J^{\eta} \phi |_2 + |u|_{L_T^{\infty}\bH^0} \omega_A(s,t)^{\frac{3}{p}} |J^{\eta} \phi|_3 + \omega_{P, \natural }(s,t)^{\frac{3}{p}} \omega_A(s,t)^{\frac{1}{p}} |J^{\eta} \phi|_4\notag \\
&\quad + \omega_{\mu}(s,t) \omega_A(s,t)^{\frac{2}{p}} |J^{\eta} \phi |_3 + |u|_{L_T^{\infty}\bH^0} \omega_A(s,t)^{\frac{3}{p}} |J^{\eta}\phi|_3 + |u|_{L_T^{\infty}\bH^0} \omega_A(s,t)^{\frac{4}{p}} |J^{\eta}\phi|_4 \notag \\
&\quad + \omega_{P, \natural }(s,t)^{\frac{3}{p}} \omega_A(s,t)^{\frac{2}{p}} |J^{\eta} \phi|_5 \notag \\
& \lesssim  \left(\omega_{\mu}(s,t) \omega_A(s,t)^{\frac{1}{p}} + |u|_{L_T^{\infty}\bH^0} \omega_A(s,t)^{\frac{3}{p}} + \omega_{P, \natural }(s,t)^{\frac{3}{p}} \omega_A(s,t)^{\frac{1}{p}} \eta^{-1}  \right.\notag  \\
&\quad \left. + \omega_{\mu}(s,t) \omega_A(s,t)^{\frac{2}{p}}+ |u|_{L_T^{\infty}\bH^0} \omega_A(s,t)^{\frac{3}{p}} + |u|_{L_T^{\infty}\bH^0} \omega_A(s,t)^{\frac{4}{p}}  \eta^{-1} \right.\notag \\
&\quad +\left.  \omega_{P, \natural }(s,t)^{\frac{3}{p}} \omega_A(s,t)^{\frac{2}{p}} \eta^{-2} \right) |\phi|_3  . \label{ineq:remainderestlastline}
\end{align}
Setting $\eta = \omega_A(s,t)^{\frac{1}{p}} \lambda $ for some constant $\lambda > 0$ to be determined later,  we have 
\begin{align*}
|\delta u^{P, \natural }_{s \theta t} |_{-3} & \lesssim   |u|_{L^{\infty}_T\bH^0}  \omega_A(s,t)^{\frac{3}{p}} ( \lambda^{-1} + 1 + \lambda + \lambda^2) + \omega_{\mu}(s,t) \omega_A(s,t)^{\frac{1}{p}} \\
&\quad + \omega_{\mu}(s,t) \omega_A(s,t)^{\frac{2}{p}}  +  \omega_{P, \natural }(s,t)^{\frac{3}{p}} (\lambda^{-1}+ \lambda^{-2}) \\
& \lesssim_{p}  \left( |u|^{\frac{p}{3}}_{L^{\infty}_T\bH^0}  \omega_A(s,t) ( \lambda^{-1} + 1 + \lambda + \lambda^2)^{\frac{p}{3}} + \omega_{\mu}(s,t)^{\frac{p}{3}} \omega_A(s,t)^{\frac{1}{3}}  \right.\\
& \quad+  \left. \omega_{\mu}(s,t)^{\frac{p}{3}} \omega_A(s,t)^{\frac{2}{3}}  +  \omega_{P, \natural }(s,t)(\lambda^{-1}+ \lambda^{-2})^{\frac{p}{3}} \right)^{\frac{3}{p}} .
\end{align*}
Applying Lemma \ref{sewingLemma}, we get
\begin{align*}
| u^{P, \natural }_{s t} |_{-3}^{\frac{p}{3}}&\lesssim_{p} |u|^{\frac{p}{3}}_{L^{\infty}_T\bH^0}  \omega_A(s,t) ( \lambda^{-1} + 1 + \lambda + \lambda^2)^{\frac{p}{3}} + \omega_{\mu}(s,t)^{\frac{p}{3}} \omega_A(s,t)^{\frac{1}{3}}  .\\
&\quad \quad +  \omega_{\mu}(s,t)^{\frac{p}{3}} \omega_A(s,t)^{\frac{2}{3}}  +  \omega_{P, \natural }(s,t)(\lambda^{-1}+ \lambda^{-2})^{\frac{p}{3}}.
\end{align*}
Since $\omega_{P, \natural}=|u^{P, \natural}|^{\frac{p}{3}}_{\frac{p}{3} -var; [s,t];\bH^{-3}}$ is equal to the infimum over all controls satisfying $|u_{st}^{P, \natural }|_{-3} \leq \omega_{P, \natural }(s,t)^{\frac{3}{p}}$ (see  \eqref{smoothingOperator}), there is a constant $C=C(p,d)$ such that 
\begin{align*}
\omega_{P, \natural }(s, t) &\le C  \left( |u|^{\frac{p}{3}}_{L^{\infty}_T\bH^0}  \omega_A(s,t) ( \lambda^{-1} + 1 + \lambda + \lambda^2)^{\frac{p}{3}} + \omega_{\mu}(s,t)^{\frac{p}{3}} \omega_A(s,t)^{\frac{1}{3}}  \right.\\
&\quad \quad\quad +  \left. \omega_{\mu}(s,t)^{\frac{p}{3}} \omega_A(s,t)^{\frac{2}{3}}  +  \omega_{P, \natural }(s,t)(\lambda^{-1}+ \lambda^{-2})^{\frac{p}{3}} \right).
\end{align*}
Choosing $\lambda$ such that $C(\lambda^{-1}+ \lambda^{-2})^{\frac{p}{3}} \leq \frac{1}{2}$ and $\tilde{L}>0$ such that $\eta=\omega_A(s,t)^{\frac{1}{p}} \lambda \leq \tilde{L}\lambda\le  1$, we obtain \eqref{RemainderEstimateMu}.

The proof of \eqref{RemainderEstimateWithoutMu} replaces the bound $ \delta u_{st}(\phi)\lesssim \omega_{\mu}(s,t)|\phi|_1$ with the bound
\begin{align}\label{ineq:othercontrol}
| \delta \mu_{st}( \phi) | & \leq \int_s^t \left(\nu |(u_r, \Delta \phi) | + |B_P(u_r)(\phi) | \right)dr  \lesssim \int_s^t \left(|u_r|_0 |\phi|_2 + |u_r|_0^2 |\phi|_{3 - \epsilon}\right) dr \\
& \lesssim (t-s) ( 1 + |u|_{L^{\infty}_T\bH^0})^2 |\phi|_{3 - \epsilon},
\end{align}
where we have used the antisymmetric property of $B_P$ and  \eqref{trilinear form estimate} with $\alpha_1 = \alpha_3 = 0$ and $\alpha_2 =3 - \epsilon$ for any $\epsilon < \frac12$. We note that this is only possible when $d \leq 3$. 
The rest of the proof is similar to the proof of \eqref{RemainderEstimateMu}. Indeed, in \eqref{ineq:remainderestlastline}, the term $\omega_{\mu}(s,t) \omega_A(s,t)^{\frac{1}{p}}$ is replaced with $ ( 1 + |u|_{L^{\infty}_T\bH^0})^2(t-s) \omega_A(s,t)^{\frac{1}{p}}\eta^{-1 + \epsilon}$ and  the term $\omega_{\mu}(s,t) \omega_A(s,t)^{\frac{2}{p}}$ is replaced with $(1 + |u|_{L^{\infty}_T\bH^0})^2(t-s) \omega_A(s,t)^{\frac{2}{p}}\eta^{-2 + \epsilon}$. Moreover, we still take $\eta = \omega_A(s,t)^{\frac{1}{p}} \lambda $ and for simplicity let $\epsilon = \frac14$.
\end{proof}
\begin{remark}
We use the estimate \eqref{RemainderEstimateWithoutMu} in the proof of existence, since it is allows us to obtain a bound independent of the Galerkin approximation.
\end{remark}
\begin{lemma} \label{AprioriVariation}
Assume that $(u,\pi)$ is a solution to  \eqref{NSDiffForm}. Then $u\in C^{p-\textnormal{var}}([0,T];\bH^{-1})$ and  there is a constant $\tilde{L}>0$, depending only on $p$ and $d$,   such that for all $(s,t)\in \Delta_T$ with $\varpi(s,t)\le L$, $\omega_{A}(s,t)\leq \tilde{L}$, and $\omega_{P,\natural}(s,t)\leq \tilde{L}$, it holds that
$$
\omega_u(s,t) \lesssim_{p} (1 + |u|_{L^{\infty}_T\bH^0})^{p} (\omega_{P, \natural }(s,t) + \omega_{\mu}(s,t)^{p} + \omega_A(s,t)), 
$$
where $\omega_u (s,t) := |u|^p_{p - \textnormal{var}; [s,t];\bH^{-1}}$.
\end{lemma}
\begin{proof}
For  all $\eta \in (0,1]$, $(s,t)\in \Delta_T$ and $\phi \in \bH^1$, we have
$$
\delta u_{st} (\phi)  = \delta u_{st}(J^{\eta} \phi)+ \delta u_{st}((I - J^{\eta})\phi) .
$$
Applying \eqref{smoothingOperator}, we find
$$
|\delta u_{st}((I - J^{\eta})\phi)| \leq 2 |u|_{L^{\infty}_T\bH^0} |(I - J^{\eta})\phi|_0 \lesssim \eta |u|_{L^{\infty}_T\bH^0} |\phi|_1 .
$$
In order to estimate the smooth part, we expand $\delta u_{st}$ using \eqref{RNS2} and then apply \eqref{smoothingOperator} to get
\begin{align*}
|\delta u_{st}(J^{\eta} \phi)| & \leq |u_{st}^{P,\natural}  (J^{\eta} \phi) | + |\delta \mu_{st} (J^{\eta} \phi)| + |u_s (A^{P,1,*}_{st} J^{\eta} \phi) |+ |u_s (A^{P,2,*}_{st} J^{\eta} \phi) | \\
& \lesssim \omega_{P, \natural }(s,t)^{\frac{3}{p}} |J^{\eta} \phi |_{3} + \omega_{\mu}(s,t)  |J^{\eta} \phi|_1 + |u|_{L^{\infty}_T\bH^0} \omega_A(s,t)^{\frac{1}{p}} | J^{\eta} \phi |_{1} + |u|_{L^{\infty}_T\bH^0} \omega_A(s,t)^{\frac{2}{p}} | J^{\eta} \phi |_{2} \\
& \lesssim \left( \omega_{P, \natural }(s,t)^{\frac{3}{p}} \eta^{-2}  + \omega_{\mu}(s,t)  + |u|_{L^{\infty}_T\bH^0} \omega_A(s,t)^{\frac{1}{p}}  + |u|_{L^{\infty}_T\bH^0} \omega_A(s,t)^{\frac{2}{p}} \eta^{-1}  \right)  |\phi |_{1} ,
\end{align*}
for all $(s,t)\in \Delta_T$ such that $\varpi(s,t)\le L$.
Setting $\eta = \omega_{P, \natural }(s,t)^{\frac{1}{p}} + \omega_A(s,t)^{\frac{1}{p}}$ and choosing $\tilde{L}>0$ such that $\eta \in (0,1]$, we get
\begin{align*}
|\delta u_{st} |_{-1} & \lesssim_{p} (1+  |u|_{L^{\infty}_T\bH^0}) \left(  \omega_{P, \natural }(s,t) + \omega_{\mu}(s,t)^{p} +   \omega_A(s,t)\right)^{\frac{1}{p}},
\end{align*}
which proves the claim.
\end{proof}

The following lemma   shows that the solution $u$ is controlled by $A^{P,1}$. 

\begin{lemma} \label{AprioriVariation1}
Assume that $(u,\pi)$ is a solution of \eqref{NSDiffForm}. Then $u^{\sharp}\in C_2^{\frac{p}{2}-\textnormal{var}}([0,T];\bH^{-2})$ and  there is a constant $\tilde{L}>0$, depending only on $p$ and $d$,   such that for all $(s,t)\in \Delta_T$ with $\varpi(s,t)\le L$, $\omega_{A}(s,t)\leq \tilde{L}$, and $\omega_{P,\natural}(s,t)\leq \tilde{L}$, it holds that
$$
\omega_\sharp(s,t) \lesssim_{p}  (1 + |u|_{L^{\infty}_T\bH^0})^{\frac{p}{2}} (\omega_{P, \natural }(s,t) + \omega_{\mu}(s,t)^\frac{p}{2} + \omega_A(s,t)), 
$$
where $\omega_\sharp (s,t) := |u^\sharp|^\frac{p}{2}_{\frac{p}{2} - \textnormal{var}; [s,t];\bH^{-2}}$.
\end{lemma}
\begin{proof}
For all  $\eta \in (0,1]$, $(s,t)\in \Delta_T$ and $\phi \in \bH^2$, we have
$$
u^\sharp_{st} (\phi)  = u^\sharp_{st}(J^{\eta} \phi)+ u^\sharp_{st}((I - J^{\eta})\phi) .
$$
We recall that we have two formulas for $u^\sharp$:
$$
u^\sharp_{st}=\delta u_{s t}-A^{P,1}_{st}u_s=\delta\mu_{st}+A^{P,2}_{st}u_s+u^{P,\natural}_{st}.
$$
As explained above, we employ the first formula to estimate the non-smooth part and the second one to estimate  the smooth part. Applying \eqref{smoothingOperator}, we find
\begin{align*}
|u^\sharp_{st}((I - J^{\eta})\phi)|&\leq|\delta u_{s t}((I - J^{\eta})\phi)|+|u_s(A^{P,1,*}_{st}(I - J^{\eta})\phi)| \\
&\le |u|_{L^{\infty}_T\bH^0}|(I - J^{\eta})\phi|_0 +|u|_{L^{\infty}_T\bH^0}\omega_A(s,t)^\frac{1}{p} |(I - J^{\eta})\phi|_1 \\
&\lesssim \Big(\eta^2 |u|_{L^{\infty}_T\bH^0} +\eta |u|_{L^{\infty}_T\bH^0}\omega_A(s,t)^\frac{1}{p} \Big)|\phi|_2 .
\end{align*}
In order to estimate  the non-smooth part, we apply  \eqref{smoothingOperator} to obtain
\begin{align*}
|u^\sharp_{st}(J^{\eta} \phi)| & \leq |u_{st}^{P,\natural}  (J^{\eta} \phi) | + |\delta \mu_{st} (J^{\eta} \phi)| +  |u_s (A^{P,2,*}_{st} J^{\eta} \phi) | \\
& \lesssim\omega_{P, \natural }(s,t)^{\frac{3}{p}} |J^{\eta} \phi |_{3} + \omega_{\mu}(s,t)  |J^{\eta} \phi|_1 +  |u|_{L^{\infty}_T\bH^0} \omega_A(s,t)^{\frac{2}{p}} | J^{\eta} \phi |_{2} \\
& \leq \left( \omega_{P, \natural }(s,t)^{\frac{3}{p}} \eta^{-1}  + \omega_{\mu}(s,t)  + |u|_{L^{\infty}_T\bH^0} \omega_A(s,t)^{\frac{2}{p}}  \right)  |\phi |_{2} ,
\end{align*}
for all $(s,t)\in \Delta_T$ with $\varpi(s,t)\le L$.
Setting $\eta = \omega_{P, \natural }(s,t)^{\frac{1}{p}} + \omega_A(s,t)^{\frac{1}{p}}$ and choosing $\tilde{L}>0$ such that $\eta \in (0,1]$, we find
\begin{align*}
| u^\sharp_{st} |_{-2}  \lesssim_{p} (1+  |u|_{L^{\infty}_T\bH^0}) \left(  \omega_{P, \natural }(s,t) + \omega_{\mu}(s,t)^\frac{p}{2} +   \omega_A(s,t)\right)^{\frac{2}{p}},
\end{align*}
which proves the claim.
\end{proof}

We now derive estimates for $\omega_{Q,\natural}$. The computation in the proof of the lemma show why \eqref{quasiChen} is the correct Chen's relation for this system.

\begin{lemma} \label{PressureRemainderLemma}
Assume that $(u,\pi)$ solves  \eqref{NSDiffForm}. For $(s,t)\in \Delta_T$ such that  $\varpi(s,t)\le L$, let $\omega_{Q, \natural }(s,t) := |u^{Q, \natural}|^{\frac{p}{3}}_{\frac{p}{3} -var; [s,t];\bH_{\perp}^{-3}}.$	
Then  there is a constant $\tilde{L}>0$, depending only on $p$ and $d$,  such that  for all $(s,t)\in \Delta_T$ with  $\varpi(s,t)\le L$ and $\omega_{A}(s,t)\leq \tilde{L}$,	
\begin{equation} \label{PressureRemainderMu}
\omega_{Q, \natural}(s,t) \lesssim_{p}|u|_{L^{\infty}_T\bH^0}^\frac{p}{3} \omega_A(s,t) + \omega_{\mu}(s,t)^{\frac{p}{3}} \omega_A(s,t)^{\frac{1}{3}}  + \omega_{P, \natural}(s,t) + \omega_u(s,t)^{\frac{1}{3}} \omega_A(s,t)^{\frac{2}{3}}.
\end{equation}
%and 
%\begin{equation} \label{PressureRemainderNoMu}
%\omega_{Q, \natural}(s,t)\lesssim |u|^{\frac{p}{3}}_{L^{\infty}_T\bH^0}  \omega_A(s,t) + ( 1 + |u|_{L^{\infty}_T\bH^0} )^{\frac{2p}{3}}(t-s)^{\frac{p}{3}}   + \omega_{P, \natural}(s,t)+ \omega_u(s,t)^{\frac{1}{3}} \omega_A(s,t)^{\frac{2}{3}} .
%\end{equation}
\end{lemma}

\begin{proof}
Applying $\delta$ to \eqref{SystemSolutionU}, we find  that for all $\psi \in \bH_{\perp}^{3}$ and $(s,\theta,t)\in \Delta_T^{(2)},$
\begin{align*}
\delta u_{s \theta t}^{Q, \natural}( \psi) & = u_{s  t}^{Q, \natural}( \psi) - u_{s \theta }^{Q, \natural}( \psi) - u_{ \theta t}^{Q, \natural}( \psi) \\
& = \delta u_{s \theta} (A_{\theta t}^{Q,1, *} \psi) + \delta u_{s \theta } (A_{\theta t}^{Q,2, *} \psi) - u_s( A_{\theta t}^{P,1, *} A_{\theta t}^{Q,1, *} \psi) \\
& =u_{s\theta}^{\sharp}( A_{\theta t}^{Q,1,*} \psi) +  \delta u_{s \theta } (A_{\theta t}^{Q,2, *} \psi),
\end{align*}
where we have used \eqref{quasiChen} in the second equality. Using Lemma \ref{AprioriVariation}, it is easy to see that  the last term  satisfies \eqref{PressureRemainderMu}% and \eqref{PressureRemainderNoMu}
, so we focus on the first term.

As usual, we split the equality into smooth and non-smooth parts $\psi = J^{\eta} \psi + (I- J^{\eta}) \psi$ for $\eta\in (0,1]$ to be determined later. In order to estimate the non-smooth part, we use  $u^{\natural}_{s\theta}= \delta u_{s \theta}-A^{P,1}_{s\theta}u_s$  and \eqref{smoothingOperator} to obtain
\begin{align*}
( \delta u_{s \theta} - A^{P,1}_{s \theta} u_s)((I- J^{\eta}) \psi) & =   \delta u_{s \theta}(A_{\theta t}^{Q,1,*}(I- J^{\eta})\psi) -  u_s( A^{P,1,*}_{s \theta}A_{\theta t}^{Q,1,*}(I- J^{\eta}) \psi) \\
& \leq 2|u|_{L^{\infty}_T\bH^0} \omega_A(s,t)^{\frac{1}{p}}| (I- J^{\eta})\psi|_1 + |u|_{L^{\infty}_T\bH^0} \omega_A(s,t)^{\frac{2}{p}}| (I- J^{\eta})\psi|_2  \\
& \lesssim |u|_{L^{\infty}_T\bH^0} (\omega_A(s,t)^{\frac{1}{p}} \eta^2 +   \omega_A(s,t)^{\frac{2}{p}} \eta) |\psi|_3.
\end{align*}
To estimate  the smooth part, we write  $u^\sharp_{st}=  \delta \mu_{s \theta } + A_{s \theta}^{P,2} u_s + u_{s \theta}^{P, \natural }$ and apply  \eqref{smoothingOperator} to get
\begin{align*}
( \delta u_{s \theta} - A^{P,1}_{s \theta} u_s)(J^{\eta} \psi) & =  \delta \mu_{s \theta}( A_{\theta t}^{Q,1,*} J^{\eta} \psi) + u_s(A_{s \theta }^{P,2,*} A_{\theta t}^{Q,1,*} J^{\eta} \psi) + u^{P,\natural}_{s \theta} ( A_{\theta t}^{Q,1,*} J^{\eta} \psi)  \\
& \lesssim \omega_{\mu}(s,t) \omega_A(s,t)^{\frac{1}{p}} |J^{\eta} \psi|_2 + |u|_{L^{\infty}_T\bH^0} \omega_A(s,t)^{\frac{3}{p}}| J^{\eta}\psi|_3 \\&\quad + \omega_{P, \natural}(s,t)^{\frac{3}{p}} \omega_A(s,t)^{\frac{1}{p}} | J^{\eta}\psi|_4 \\
& \lesssim (\omega_{\mu}(s,t) \omega_A(s,t)^{\frac{1}{p}}  + |u|_{L^{\infty}_T\bH^0} \omega_A(s,t)^{\frac{3}{p}} + \omega_{P, \natural}(s,t)^{\frac{3}{p}} \omega_A(s,t)^{\frac{1}{p}} \eta^{-1}) |\psi|_3  .
\end{align*}
Setting $\eta = \omega_A(s,t)^{\frac{1}{p}}$ and choosing $\tilde{L}$ such that $\eta \in (0,1]$, we get
\begin{align*}
|\delta u^{Q, \natural}_{s \theta t}|_{-3} \lesssim\left(|u|_{L^{\infty}_T\bH^0} \omega_A(s,t)^{\frac{3}{p}} + \omega_{\mu}(s,t) \omega_A(s,t)^{\frac{1}{p}}  + \omega_{P, \natural}(s,t)^{\frac{3}{p}} + \omega_u(s,t)^{\frac{1}{p}} \omega_A(s,t)^{\frac{2}{p}}\right)\\
\lesssim_{p}\left(|u|_{L^{\infty}_T\bH^0}^\frac{p}{3} \omega_A(s,t) + \omega_{\mu}(s,t)^{\frac{p}{3}} \omega_A(s,t)^{\frac{1}{3}}  + \omega_{P, \natural}(s,t) + \omega_u(s,t)^{\frac{1}{3}} \omega_A(s,t)^{\frac{2}{3}}\right)^{\frac{3}{p}}.
\end{align*}
Using Lemma \ref{sewingLemma}, we obtain the first inequality. The proof of the second inequality is similar to the first; see the end of the proof of Lemma  \ref{Thm2.5}.
\end{proof}

By virtue of Lemma \ref{PressureRemainderLemma} and \eqref{SystemSolutionPi}, we see immediately that $\pi \in C^{p - var}([0,T]; \bH_{\perp}^{-3})$, although we conjecture that there is  better spatial regularity.

\section{Proof of the main results}
\label{main}

\subsection{Existence, proof of Theorem \ref{existenceThmNoGalerkin}}
\label{Section:Galerkin}

\subsubsection{Galerkin approximation}
\label{s:galerkin}

We prove the existence of a solution using a Galerkin approximation.
Let $\{h_n\}_{n=1}^{\infty}$ be the eigenfunctions of the Stokes operator with corresponding eigenvalues $\{\lambda_n\}_{n=1}^{\infty}$. As discussed in Section \ref{ss:notation}, the collection $\{h_n\}_{n=1}^{\infty}$ is an orthonormal basis of $\bH^0$ and an orthogonal basis of $\bH^1$.
For a given $N\in \mathbf{N}$, let
$\bH_N=\operatorname{span}(\{h_1,\ldots,h_N\})$ and  $P_N: \bH^{-1}\rightarrow \bH_N$ be defined by 
$$P_N\varv:=\sum_{n=1}^N(\varv,h_n) h_n, \quad \varv\in \bH^{-1}.$$
Since $\mathbf{Z}\in C_g^{p-\textnormal{var}}([0,T];\bR^K)$ is a geometric rough path, there is a sequence  of $\bR^K$-valued smooth paths  $\{z^{N}\}_{N=1}^{\infty}$ such that their canonical lifts $\bZ^N=(Z^N,\mathbb{Z}^N)$ converge to $\bZ$ in the rough path topology.  We assume that
\begin{equation} \label{uniformRPBounds}
|Z_{st}^{N} | \lesssim \omega_Z(s,t)^{\frac{1}{p}},\quad |\mathbb{Z}^{N}_{st}| \lesssim \omega_Z(s,t)^{\frac{2}{p}}, \quad \forall (s,t)\in \Delta_T.
\end{equation}
For convenience, let us denote by $N_0$ a constant that bounds  $\sigma=(\sigma_1,\ldots,\sigma_K)$ and its derivatives up to order two.

We consider the following $N$-th order Galerkin approximations of \eqref{NSDiffForm}:
\begin{equation}\label{eq:Smooth NS}
\partial_t u^N + P_NB_P(u^N) =  \nu P_N \Delta u^N+  \sum_{k=1}^{K} P_N P[(\sigma_k\cdot \nabla) u^N]  \dot{z}^{N,k}_t,
\end{equation}
where $u^N(0)=P_N u_0$. If we assume that 
$$
u^N(t,x)=\sum_{i=1}^Nc^N_i(t)h_i(x),
$$
then plugging in this expansion for $u^N(t,x)$ into \eqref{eq:Smooth NS} and  testing against $h_i$,  we derive an ODE for the coefficients $(c_i^N)_{i=1}^N$: 
\begin{equation}\label{eq:ODE system}
\dot{c}^N_i(t)+\sum_{j,l=1}^NB_{j,l,i}c^N_j(t)c^N_l(t)=\nu\lambda_ic^N_i(t)+\sum_{k=1}^K\sum_{j=1}^N  A_{k,j,i}c^N_j(t)\dot{z}^{N,k}_t ,
\end{equation}
where  $B_{j,l,i}:=P_NB_P(h_j,h_l)(h_i)$ and $A_{k,j,i}:=((\sigma_k\cdot \nabla )h_j)(h_i)$. Owing to   \eqref{trilinear form estimate} with $\alpha_1=1,\alpha_2=0$, and $\alpha_3=1$, for all $j,l$ and $i$,  we have  $$|B_{j,l,i}|\le |h_j|_1|h_l|_1|h_i|_1.$$ Moreover,   for all $k,j,$ and $i$, $$ |A_{k,j,i}|\le |\sigma_k|_{\bL^{\infty}}|h_j|_1 |h_i|_0.$$ Thus, \eqref{eq:ODE system} has locally Lipschitz coefficients,  and so there exists a unique solution $(c_i)_{i=1}^N$ of \eqref{eq:ODE system} on a time interval $[0,T_N)$, for some $T_N>0$.  Therefore, $
u^N_t(x)=\sum_{i=1}^Nc^N_i(t)h_i(x)$   is a solution of \eqref{eq:Smooth NS} on the time  interval $[0,T_N)$. 

To get a global solution, we derive a global energy estimate of $u^N$. Testing \eqref{eq:Smooth NS} against $u^N$  and using \eqref{eq:B prop}, the divergence theorem, and that  $\nabla \cdot\sigma_k=0$, for all $k\in \{1,\ldots,K\}$, we get 
\begin{align*}
|u^N_t|_0^2 +  2 \nu\int_0^t |\nabla u^N_s|_0^2\, ds& =  |P_N u_0|_0^2 - 2 \int_0^t P_N B_P(u^N_s,u^N_s,u^N_s)\,ds  + \sum_{k=1}^K \int_0^t ((\sigma_k \cdot\nabla) u^N_s, u^N_s) \dot{z}_s^{N,k}ds  \\
&= |P_Nu_0|_0^2 \leq |u_0|_0^2, \quad \forall t\in [0,T_N).
\end{align*}
It follows that the $\bL^2$-norm of $u^N$ is non-increasing in time, and hence that $(c_i)_{i=1}^N$ does not blow-up in finite time. Therefore,  $u^N\in C_T\bH^0\cap L^2_T\bH^1$ solves \eqref{eq:Smooth NS}.

Integrating \eqref{eq:Smooth NS}  over $[s,t]$, and then iterating  the equation into the integral against $\dot{z}^{N}$ as we did in \eqref{eq:iterationP} , we find
\begin{equation} \label{eq:URDGallerkinApprox}
\delta u^N_{st} =  \int_s^t \left( \nu P_N \Delta  u^N_r - P_N B_P(u_r^N) \right) \, dr +  A_{st}^{N,1} u_s^{N} +  A_{st}^{N,2} u_s^{N} + u_{st}^{N, \natural} ,
\end{equation}
where    $\tilde{P}_N := P_N P$,
$$
A^{N,1}_{st} \phi :=    \tilde{P}_N  \left[(\sigma_k \cdot \nabla) \phi \right]  Z_{st}^{N,k},
$$
$$
A^{N,2}_{st} \phi  :=   \tilde{P}_N \left[ (\sigma_k \cdot  \nabla)  \tilde{P}_N [ (\sigma_j \cdot  \nabla)  \phi ] \right] \mathbb{Z}_{st}^{N, j,k},
$$
$\mu_t^N : =  P_N \int_0^t \left( \nu \Delta u_r^N  -  B_P(u_r^N) \right) \, dr $, and
\begin{align*}
u_{st}^{N, \natural} & :=   \int_s^t \tilde{P}_N  \left[ (\sigma_k \cdot \nabla) \delta \mu_{sr}^N \right]  \dot{z}_r^{N,k} \, dr + \tilde{P}_N \int_s^t \int_s^r    (\sigma_k \cdot \nabla) \tilde{P}_N\left[ (\sigma_i \cdot \nabla) \delta \mu^N_{s r_{1}} \right] \dot{z}_{r_1}^{N,i} \dot{z}_r^{N,k} \, dr_1 \, dr \\
&\quad +   \int_s^t \int_s^r \int_s^{r_1}\tilde{P}_N\left[(\sigma_k\cdot  \nabla )\tilde{P}_N \left[ (\sigma_i \cdot  \nabla )\tilde{P}_N \left[( \sigma_j \cdot  \nabla) u_{r_2}^N \right]\right] \right] \dot{z}_{r_2}^{N,j} \dot{z}_{r_1}^{N,i} \dot{z}_r^{N,k} \, dr_2 \, dr_1 \, dr,
\end{align*}

Owing to   \eqref{A1bound},   \eqref{A2bound}, and \eqref{uniformRPBounds}, we have that $(A^{N,1}, A^{N,2} )$ is uniformly bounded in $N$ as a family of unbounded rough drivers on the scale $( \bH^{\alpha} )_{\alpha \in \bR_+}$. That is, there exists a control $\omega_{A^N}$ such that \eqref{ineq:UBRcontrolestimates} holds and  for all $(s,t)\in \Delta_T$,
$$
\omega_{A^N}(s,t)\lesssim_{N_0}  \omega_{Z}(s,t).
$$

It is straightforward to check that    $u^{N, \natural} \in C_2^{\frac{p}{3} -\textnormal{var}}([0,T]; H_N)$ by estimating term-by-term; one makes use of  \eqref{A1bound},  \eqref{A2bound},  \eqref{trilinear form estimate}, and that $u^N$ is smooth in space and $z^N$ is smooth in time. For all $(s,t)\in \Delta_T$, let $\omega_{N,\natural}(s,t) := |u^{N,\natural}|^{\frac{p}{3}}_{\frac{p}{3} -\textnormal{var};[s,t];\bH^{-3}}$.
Arguing as in Lemma \ref{Thm2.5}, we find that there is an $L>0$ such that for all $(s,t)\in \Delta_T$ with  $\omega_{Z}(s,t)\le L$, 
\begin{align}
\omega_{N, \natural}(s,t)  & \lesssim_p|u^N|^{\frac{p}{3}}_{L_T^{\infty} H_N} \omega_{A^N}(s,t)  + ( 1 +|u^N|_{L_T^{\infty} H_N})^{\frac{2p}{3}}  (t-s)^{\frac{p}{3}} \omega_{A^N}(s,t)^{\frac{1}{12}}  \notag \\
& \lesssim_{p,N_0} |u_0|_{0}^{\frac{p}{3}} \omega_{Z}(s,t)  +  ( 1 + |u_0|_{0})^{\frac{2p}{3}}(t-s)^{\frac{p}{3}} . \label{uniformRemainderBounds}
\end{align}

\begin{theorem} \label{existenceThm}
There exists a subsequence of $\{ u^N \}_{N=1}^{\infty}$ that converges  weakly in $L^2_T\bH^1$, weak-* in $L^{\infty}_T\bH^0$, and strongly in $L_T^{2}\bH^0 \cap C_T\bH^{-1}$  to a solution of \eqref{SystemSolutionU} that is weakly continuous in $\bH^0$.
\end{theorem}
\begin{proof}
Since $\{u^N\}_{N=1}^{\infty}$ is uniformly bounded in $L^2_T\bH^1\cap L^{\infty}_T\bH^0$, an application of Banach-Alaoglu yields a subsequence, which we will relabel as $\{u^N\}_{n=1}^{\infty}$, that converges   weakly in $L^2_T\bH^1$ and weak-* in $L^{\infty}_T\bH^0$. To obtain a further subsequence that converges strongly in $L_T^{2}\bH^0 \cap C_T\bH^{-1}$, we need to apply  Lemma \ref{CompactnessLemma}; that is, we need to show there exists a controls $\omega$ and $\bar{\omega}$ and $L,\kappa>0$  independent of $N$ such that  $|\delta u_{st}^N|_{-1}\le \omega(s,t)$ for all $(s,t)\in \Delta_T$ with $\bar{\omega}(s,t)\le L$. The proof of this is similar to the proof of Lemma \ref{AprioriVariation},  except that we need a slightly different bound on the drift term. This bound,  in particular, does not yield $p$-variation of the solution.

Let $\phi \in \bH^1$. Decomposing $\delta u^N_{st}$ into a smooth and non-smooth part using $J^{\eta}$ for some $\eta\in (0,1]$, we get
\begin{align*}
|\delta u^N_{st}(\phi)| & \leq |\delta u^N_{st}( J^{\eta} \phi)| + |\delta u^N_{st}( (I - J^{\eta}) \phi)| \\
& \lesssim  \omega_{N, \natural}(s,t)^{\frac{3}{p}} |J^{\eta} \phi |_3 + (t-s)(1  + |u^N|_{L_T^{\infty}H_N})^2 |J^{\eta} \phi |_3 \\
& \quad+ |u^N|_{L_T^{\infty}H_N} (\omega_{A^N}(s,t)^{\frac{1}{p}} |\phi|_1 + \omega_{A^N}(s,t)^{\frac{2}{p}} |J^{\eta}\phi|_2 ) + |u^N|_{L_T^{\infty}H_N} |(I - J^{\eta}) \phi|_0 \\
& \lesssim  \eta^{-2}\omega_{N, \natural}(s,t)^{\frac{3}{p}} |\phi |_1 +\eta^{-2} (t-s)  (1  + |u_0|_0)^2 |\phi |_1 \\
&\quad + |u_0|_{0} (\omega_{Z}(s,t)^{\frac{1}{p}}  +  \eta^{-1}\omega_{Z}(s,t)^{\frac{2}{p}})  |\phi|_1  + \eta|u_0|_{0}  |\phi|_1 .
\end{align*}
Using \eqref{uniformRemainderBounds} together with $\eta = \omega_{Z}(s,t)^{\frac{1}{p}} + (t-s)^{\frac{1}{p}} $ and $L > 0 $ chosen such that $\eta \in (0,1]$, we find 
\begin{align}
|\delta u^N_{st}|_{-1} & \le (1+ |u_0|_0)^2  \left[ \big( \omega_{Z}(s,t)^{ \frac{3}{p}} +  (t-s)\omega_{Z}(s,t)^{\frac{1}{p}}  \big) \eta^{-2}  \right. \notag \\
&  \left. \quad+ (t-s) \eta^{-2}  + (\omega_{Z}(s,t)^{\frac{1}{p}}  + \omega_{Z}(s,t)^{\frac{2}{p}} \eta^{-1})   + \eta \right] \notag \\
& \lesssim_{N_0}(1+ |u_0|_0)^2 ( \omega_{Z}(s,t)^{ \frac{1}{p} } + (t-s)^{1 - \frac{2}{p}} )  \label{uniformSolution}  .
\end{align}
By Lemma \ref{CompactnessLemma}, there is a subsequence of  $\{u^{N}\}_{N=1}^{\infty}$ which we continue to denote by $\{u^{N}\}_{N=1}^{\infty}$ converging strongly to an element $u $ in $C_T\bH^{-1}\cap L^2_T\bH^0$. Furthermore, owing to Lemma \ref{lem:weakc}, we know that $u$ is continuous with values in $\bH^0_w$ (i.e., $ \bH^0$ equipped with the weak topology). 

Our goal now is to pass to the limit in \eqref{eq:URDGallerkinApprox} tested against some $\phi \in \bH^3$ as $N$ tends to infinity. Clearly, 
\begin{align*}
| (A_{st}^{N,1} -  A_{st}^{P,1})  \phi |_{0} &=       \left| P_N  P \left[\sigma_k \cdot \nabla) \phi\right]   Z_{st}^{N,k} -   P \left[(\sigma_k \cdot \nabla) \phi\right]   Z_{st}^{k}  \right|_{0} \\
& \leq | P_N  P \left[(\sigma_k \cdot \nabla) \phi\right]  ( Z_{st}^{N,k} -    Z_{st}^{k} ) |_{0}  +  | (I - P_N)  P \left[(\sigma_k \cdot \nabla) \phi\right]   Z_{st}^{k} |_{0}.
\end{align*}
Making use of \eqref{A1bound}, we get
\begin{gather*}
| P_NP(\sigma_k \cdot \nabla) \phi  |_{0} | Z_{st}^{N,k} -    Z_{st}^{k} | \lesssim_{N_0} | \phi  |_{1} | Z_{st}^{N} -    Z_{st}| ,\\
| (I - P_N)  P \left[(\sigma_k \cdot \nabla) \phi\right]   Z_{st}^{k} |_{0}\le |I-P_N|_{\clL(\bH^0,\bH^0)}\left|P \left[(\sigma_k \cdot \nabla) \phi\right]\right|_0 |Z_{st}|\lesssim_{N_0} |I-P_N|_{\clL(\bH^0,\bH^0)}|\phi|_1|Z_{st}|.
\end{gather*}
Moreover,  we have
\begin{align*}
| (A_{st}^{N,2} -  A_{st}^{P,2})  \phi |_{0} & \leq |  \tilde{P}_N \left[  (\sigma_k \cdot  \nabla ) \tilde{P}_N [ (\sigma_j \cdot  \nabla)  \phi ] \right] (\mathbb{Z}_{st}^{N, j,k} - \mathbb{Z}_{st}) |_{0} \\
&\quad +  | (I -  P_N) P \left[  (\sigma_k \cdot  \nabla)  P [ (\sigma_j \cdot  \nabla)  \phi ] \right] \mathbb{Z}_{st}^{j,k} |_{0} \\
& \quad + | \tilde{P}_N  \left[  (\sigma_k \cdot  \nabla )( I - P_N)  P [ (\sigma_j \cdot  \nabla)  \phi ] \right] \mathbb{Z}_{st}^{j,k} |_{0} .
\end{align*}
Now, applying \eqref{A2bound}, we find
\begin{gather*}
|  \tilde{P}_N \left[  (\sigma_k \cdot  \nabla ) \tilde{P}_N [ (\sigma_j \cdot  \nabla)  \phi ] \right] (\mathbb{Z}_{st}^{N, j,k} - \mathbb{Z}_{st}^{j,k}) |_{0}\lesssim_{N_0} |\phi|_2|\mathbb{Z}_{st}^{N, j,k} - \mathbb{Z}_{st}^{j,k}|,\\
| (I -  P_N) P \left[  (\sigma_k \cdot  \nabla)  P [ (\sigma_j \cdot  \nabla)  \phi ] \right] \mathbb{Z}_{st}^{j,k} |_{0}\lesssim_{N_0}  |I-P_N|_{\clL(\bH^0,\bH^0)}|\phi|_2|\mathbb{Z}_{st}^{j,k} |,\\
|   \tilde{P}_N(\sigma_k \cdot  \nabla) ( I - P_N)  P [ (\sigma_j \cdot  \nabla)  \phi ] ) \mathbb{Z}_{st}^{j,k} |_{0} 
\lesssim_{N_0}   |I-P_N|_{\clL(\bH^1,\bH^1)}|\phi|_2|\mathbb{Z}_{st}^{j,k}|.
\end{gather*}
Therefore, $$A_{st}^{N,i,*}\phi \rightarrow A_{st}^{P,i,*}\phi$$ in $\bH^{0}$ for $i\in \{1,2\}$ as $N \rightarrow \infty$, and hence  
\begin{align*}
|(u^N_s,A_{st}^{N,i,*}\phi) - (u_s,A_{st}^{P,i,*}\phi)|&\le_{N_0}  | (u^N_s-u_s,A_{st}^{N,i,*}\phi) - (u_s,(A_{st}^{P,i,*}-A_{st}^{N,i,*})\phi)|\\
&\lesssim_{N_0}  |u^N_s-u_s|_{-1}|\phi|_3+|u_s|_{0}|(A_{st}^{P,i,*}- A_{st}^{N,i,*})\phi|_{0}\rightarrow 0
\end{align*}
as $N \rightarrow \infty$.
Finally, using the  strong convergence in $L_T^2\bH^0$ of $\{u^N\}$ and \eqref{trilinear form estimate},  we find
\begin{align*}
\left|\int_s^t \left[B_P(u_r)(\phi) - B_P(u_r^N) (\phi) \right]\,dr \right|    &\leq \left| \int_s^t B_P(u_r - u_r^N,u_r)(\phi)\,dr \right| 
\quad + \left|\int_s^t  B_P(u_r^N,u_r- u_r^N) (\phi) \, dr \right| \\
&\lesssim  \int_s^t |u_r - u_r^N|_0 |u_r|_0 \, dr |\phi|_3    +  \int_s^t |u_r - u_r^N|_0 |u^N_r|_0 \, dr |\phi|_3 \rightarrow 0
\end{align*}
as $N \rightarrow \infty$.

Since all of the terms in equation \eqref{eq:URDGallerkinApprox} converge when applied to $\phi$,  the remainder $u^{N, \natural}_{st}(\phi)$ converges to some limit $u^{P, \natural}_{st}(\phi)$. Owing to the uniform bound \eqref{uniformRemainderBounds}, we have $ u^{P,\natural} \in C^{\frac{p}{3}-\textnormal{var}}_{2, \varpi,L}([0,T]; \bH^{-3})$ for some control $\varpi$ depending only on $\omega_Z$ and $L> 0$ depending only on $p$, which proves that $u$ is a solution of \eqref{SystemSolutionU}.
\end{proof}

\subsubsection{Pressure recovery}
\label{s:pressure}

To finalize the proof of existence, we need to prove that  the pressure term $\pi$ exists and satisfies \eqref{SystemSolutionPi}.
To this end,  we first show that we can construct the rough integral $$I_t=Q\int_0^t (\sigma_k \cdot\nabla) u_r \, d Z^k_r, \quad I_0=0,$$ using the sewing lemma, Lemma \ref{sewingLemma}. Let
$
h_{st} = A^{Q,1}_{st}u_s+ A^{Q,2}_{st}u_s
$
for $(s,t)\in \Delta_T$. It follows that   $h\in C_2^{p-\textnormal{var}}([0,T];\bH_{\perp}^{-2})$.
Applying the $\delta$ operator to $h$ and using \eqref{quasiChen}, for $(s,\theta,t)\in \Delta^{(2)}_T$, we have
\begin{align*}
\delta h_{s\theta t}&=  (\delta A^{Q,2}_{s\theta t})u_s - A^{Q,1}_{\theta t} \delta u_{s\theta}-A^{Q,2}_{\theta t}\delta u _{s\theta} \\
&=  A_{ \theta t}^{Q,1} A_{ s \theta }^{P,1}  u_s - A^{Q,1}_{\theta t} \delta u_{s\theta}-A^{Q,2}_{\theta t}\delta u _{s\theta} \\
&=-A^{Q,1}_{\theta t}u^\sharp_{s\theta}-A^{Q,2}_{\theta t} \delta u_{s\theta},
\end{align*}
where we recall that $u^{\sharp}_{st}=\delta u_{s t}-A^{P,1}_{st}u_s$ (see \eqref{eq:defin of sharp}). Owing to Lemmas \ref{AprioriVariation} and  \eqref{AprioriVariation1}, which establish the regularity of $\delta u$ and $u^\sharp$,  there are controls $\omega$  and $\varpi$ and an $L>0$  such that for all $(s,\theta,t)$ with $\varpi(s,t)\le L$, we have
$$
|\delta h_{s\theta t}|_{-3}\lesssim_p  \left(\omega_A(s,t)^{\frac{1}{3}}\omega_{\sharp}(s,t)^{\frac{2}{3}}+\omega_A(s,t)^{\frac{2}{3}}\omega_{u}(s,t)^{\frac{1}{3}}\right)^{\frac{3}{p}}=:\omega(s,t)^{\frac{3}{p}}.
$$
Therefore, by Lemma \ref{sewingLemma},  there exists a unique path $I\in C^{p-\textnormal{var}}([0,T];\bH^{-3}_{\perp})$ and two-index map $ I^{\natural}\in  C_{2,\varpi,L}^{p-\textnormal{var}}([0,T];\bH_{\perp}^{-3})$ such that 
$$
\delta I_{st}= A^{Q,1}_{st}u_s+ A^{Q,2}_{st}u_s + I^\natural_{st} . 
$$
and
$$| I^\natural_{s\theta t}|_{-3}\lesssim_p \omega(s,t)^\frac{3}{p}.$$  

We define
$$
\pi_t :=  -  \int_0^t B_Q(u_r) \, dr   + I_t ,
$$
or alternatively using the local approximation
$$
\delta \pi_{st} =  -  \int_s^t B_Q(u_r)\,dr   + A^{Q,1}_{st}u_s+ A^{Q,2}_{st}u_s  + u^{Q,\natural}_{st},
$$
where $u^{Q,\natural}_{st}:=I^\natural_{st}$. 
Owing to Lemma \ref{PressureRemainderLemma} and \eqref{SystemSolutionPi}, we have that $\pi \in C^{p - var}([0,T]; \bH_{\perp}^{-3})$.

\subsection{Uniqueness in two spatial dimensions, proof of Theorem \ref{contractiveTheoremNoContraction}}
\label{s:uniq}

The objective of this section is to prove that the solution $(u,\pi)$ of  \eqref{NSDiffForm} is unique  when $d=2$ and  $\sigma_k$ is constant function of $x\in \bT^d$ for all $k\in \{1,\dots, K\}$.
Assume for a moment that all functions are smooth  and that we have two solutions of \eqref{NSDiffForm}:
$$
\partial_t u^i_t = \nu \Delta u^i_t -P(u_t^i\cdot \nabla)u_t^i + P (\sigma_k \cdot \nabla) u^i_t  \dot{z}_t^k , \;\;i\in \{1,2\}. 
$$
Then $\varv := u^1 - u^2$  satisfies
$$
\partial_t \varv =  \nu \Delta \varv - (B_P(u^1) - B_P(u^2)) + P  ( \sigma_k \cdot \nabla)\varv \dot{z}_t^k ,
$$
and the chain rule gives for all $x\in \bT^2$,
$$
\frac{1}{2}\partial_t |\varv(x)|^2  =  \nu \varv(x)\cdot   \Delta \varv(x) -  \varv(x) \cdot ( B_P(u^1(x)) - B_P(u^2(x))) +  \varv(x)\cdot  (\sigma_k \cdot \nabla) \varv(x)  \dot{z}_t^k .
$$
One could proceed by integrating with respect to $x$ to obtain uniqueness and energy estimates. However, in the rough case, many of our objects are distributions, and so the action of integrating with respect to $x$ is actually applying a distribution to a test function. 

Since we do not expect our solution to be regular enough to perform this operation, we shall employ a doubling of the variables trick; that is, we consider $ t \mapsto \varv_t^{\otimes 2}(x,y) := \varv_t(x) \varv_t(y)^T$, where $T$ denotes the transpose. This is a well defined operation for any distribution and we get the formula for the square by testing this distribution against an approximation of the Dirac-delta in $x=y$. We remark that one cannot directly use  the techniques from \cite{DeGuHoTi16}, since this way of approximating the Dirac-delta violates the divergence-free condition.

Let $u^1$ and $u^2$ be solutions of \eqref{NSDiffForm}, as defined by Definition \ref{def:solution2}. For all $\phi \in \bH^3$ and $i\in\{1,2\}$ and $(s,t)\in \Delta_T$, we have
\begin{align*}
\delta u^i_{st}(\phi) = \delta\mu^i_{st}(\phi) + u^i_s( [A^{P,1,*}_{st}+A^{P,2,*}_{st} ]\phi) + u_{st}^{i;P,\natural}(\phi) ,
\end{align*}
where 
$$
\mu^i_t(\phi)=  - \int_0^t \left[  \nu (\nabla u^i_r , \nabla  \phi)  + B_P(u^i_r)(\phi)\right]\,dr.
$$
Setting $\varv=u^1-u^2$, $\varv^{\natural}=u^{1;P,\natural}-u^{2;P,\natural}$ and $\mu_t(\phi) =  - \int_0^t [ \nu ( \nabla \varv_r , \nabla \phi) + (B_P(u_r^1) - B_P(u_r^2)) (\phi) ] \, dr$, we have 
$$
\delta \varv_{st}(\phi) = \delta\mu_{st}(\phi) +  \varv_s( [A^{P,1,*}_{st}+A^{P,2,*}_{st} ]\phi) + \varv^{\natural}_{st} .
$$
Define 
$$
\omega_{\mu}(s,t)=\omega_{\mu^1}(s,t)+\omega_{\mu^2}(s,t),
$$
and notice that 
$$
\left|\delta \mu_{st}(\phi)\right|_{-1}\lesssim \omega_{\mu}(s,t).
$$

We denote by $a \st b$ the symmetrization of the tensor product of two functions $a,b:\bT^2\to\bR^2$; that is, 
$$a \st b (x,y):= \frac{1}{2} ( a \otimes b  + b \otimes a)(x,y)=\frac{1}{2} \big( a(x)  b(y)^T  + b(x) a(y)^T\big), \;\;(x,y)\in \bT^2.$$

\begin{lemma} \label{tensorLemma}
The weakly continuous mapping $\varv_t^{\otimes 2}:[0,T]\rightarrow \bH^{-3}_x \otimes \bH_y^{-3}$ satisfies the equation 
\begin{equation} \label{eq:tensor difference 2}
\delta \varv^{\otimes 2}_{st} - 2\int_s^t \left( \nu  \varv_r \st \Delta \varv_r - \varv_r \st (B_P(u^1_r) - B_P(u^2_r)) \right) \, dr =  \left(\Gamma_{st}^1+\Gamma_{st}^2\right)\varv_s^{\otimes 2} + \varv_{st}^{\otimes 2,\natural}, 
\end{equation}
where 
$$
\Gamma^1:=A^{P,1}\otimes I + I\otimes A^{P,1}, \quad \Gamma^2:= A^{P,2}\otimes I + I \otimes A^{P,2}+A^{P,1}\otimes A^{P,1},
$$
and  $v^{\otimes 2, \natural} \in C^{ \frac{p}{3}  -\textnormal{var}}_{2, \varpi, L}([0,T]; \bH^{-3}_x \otimes \bH_y^{-3})$, for a control $\varpi$ and $L>0$. 
\end{lemma}

\begin{proof}
Elementary algebraic manipulations yield
\begin{align*}
\delta v^{\otimes 2}_{st} & = 2 \varv_{s} \st \delta \varv_{st}+ \delta \varv_{st}\otimes \delta \varv_{st} = 2 \varv_s \st \varv_{st}^{\natural} + 2\varv_s \st \delta \mu_{st} + 2\varv_s  \st A^1_{st} \varv_s + 2\varv_s  \st A^{P,2}_{st} \varv_s \\
& \quad + ( \varv_{st}^{\natural} + \delta\mu_{st} + A^{P,2}_{st}\varv_s )^{\otimes 2} + 2 ( \varv_{st}^{\natural} + \delta \mu_{st} + A^{P,2}_{st}\varv_s ) \st A^{P,1}_{st}\varv_s + A^{P,1}_{st}\varv_s \otimes A^{P,1}_{st}\varv_s.
\end{align*}
Thus, 
\begin{equation} \label{eq:tensor difference}
\delta v^{\otimes 2}_{st} -  2 \int_s^t\left[ \nu  \varv_r \st \Delta \varv_r - \varv_r \st (B_P(u^1_r) - B_P(u^2_r)) \right]\,dr =  \left(\Gamma_{st}^1+\Gamma_{st}^2\right)\varv_s^{\otimes 2} + \varv_{st}^{\otimes 2,\natural} ,
\end{equation}
where 
\begin{align} \label{eq:tensorEqn2}
\varv_{st}^{\otimes 2,\natural} & :=  - 2 \int_s^t \delta \varv_{sr} \st [\nu  \Delta \varv_r + (B_P(u^1_r)-B_P(u^2_r)] \, dr + 2 \varv_{st}^{\natural} \st \varv_s \notag \\
&\quad + ( \varv_{st}^{\natural} + \delta \mu_{st} + A^{P,2}_{st}\varv_s)^{\otimes 2} + 2 ( \varv_{st}^{\natural} + \delta \mu_{st} + A^{P,2}_{st}\varv_s) \st A^{P,1}_{st}\varv_s \notag  \\
&  = - 2 \int_s^t \nu \delta \varv_{sr} \st  \Delta \varv_r dr  + 2 \int_s^t \delta \varv_{sr} \st \left[B_P(u^1_r)-B_P(u^2_r)\right]\,dr + 2 \varv_{st}^{\natural} \st \varv_s \notag \\
& \quad+  \varv_{st}^{\natural} \otimes \varv_{st}^{\natural} +  \varv_{st}^{\natural} \st \delta \mu_{st} + \varv_{st}^{\natural} \st A^{P,2}_{st}\varv_s  + \delta\mu_{st} \otimes \delta\mu_{st}  + \delta\mu_{st}\st A^{P,2}_{st}\varv_s + A^{P,2}_{st}\varv_s \otimes A^{P,2}_{st}\varv_s \notag \\
& \quad+ 2  \varv_{st}^{\natural} \st A^{P,1}_{st}\varv_s + \delta \mu_{st}\st A^{P,1}_{st}\varv_s + A^{P,2}_{st}\varv_s \st A^{P,1}_{st}\varv_s.
\end{align}
Estimating $\varv_{st}^{\otimes 2,\natural}$ term-by-term and making use of  \eqref{A1bound},  \eqref{A2bound},  and \eqref{trilinear form estimate}, we find that there is a control $\varpi$ and $L>0$ such that    $v^{\otimes 2, \natural} \in C^{ \frac{p}{3}  -\textnormal{var}}_{2, \varpi, L}([0,T]; \bH^{-3}_x \otimes \bH_y^{-3})$.
\end{proof}

Let $\mathbf{f}_n$ be the orthonormal basis of $\{u\in L^2(\bT^d;\bC^d): \nabla \cdot u=0\}$  described in Section \ref{ss:notation}.  Define $F_N(x,y) := \sum_{|n| \leq N} \mathbf{f}_n(x) \otimes \overline{\mathbf{f}_n(y)}$. It follows that for all $f,g \in \bH^0$, 
$$
f \otimes g(F_N) = \sum_{ |n| \leq N} (f,\mathbf{f}_n) \overline{(g,\mathbf{f}_n)} \rightarrow (f,g)
$$ 
as $N \rightarrow \infty$. In particular,  $\varv^{\otimes 2}(F_N)\rightarrow |\varv|_0^2$ as $N\rightarrow \infty$.
Moreover, by \eqref{basis prop}, we have 
$$\nabla_x F_N + \nabla_y F_N = 0.$$  Motivated by this, we will test  equation \eqref{eq:tensor difference} $F_N$ and pass to the limit as $N\rightarrow \infty$ to derive the equation for the square.

Because $\sigma_k$ is constant, we have
\begin{align*}
\Gamma_{st}^{1,*} F_N & =  ( (\sigma_k \cdot \nabla_x) F_N  + (\sigma_k \cdot \nabla_y) F_N ) Z_{st}^k = 0
\end{align*}
and
\begin{align*}
\Gamma_{st}^{2,*} F_N & =   (\sigma_k  \cdot \nabla_x) (\sigma_j \cdot \nabla_x) F_N \mathbb{Z}_{st}^{j,k}    + (\sigma_k  \cdot \nabla_y)  (\sigma_j \cdot \nabla_y) F_N  \mathbb{Z}_{st}^{j,k}  +  (\sigma_k  \cdot \nabla_x) ( \sigma_j \cdot \nabla_y) F_N  Z_{st}^j Z_{st}^k \\
& = (\sigma_k  \cdot \nabla_x) ( \sigma_j \cdot \nabla_x) F_N \mathbb{Z}_{st}^{j,k}    + (\sigma_k  \cdot \nabla_x) ( \sigma_j \cdot \nabla_x) F_N   \mathbb{Z}_{st}^{k,j}  -  (\sigma_k  \cdot \nabla_x) ( \sigma_j \cdot \nabla_x) F_N  Z_{st}^j Z_{st}^k \\
&  = 0,
\end{align*}
where we have used $(\sigma_k \cdot \nabla) ( \sigma_j \cdot \nabla ) = (\sigma_j \cdot \nabla) ( \sigma_k \cdot \nabla )$ and $\mathbb{Z}_{st}^{j,k} +\mathbb{Z}_{st}^{k,j}  = Z_{st}^j  Z_{st}^k$. Applying the divergence theorem, we get 
\begin{align*}
\int_s^t \nu \varv_r \otimes \Delta \varv_r (F_N) \, dr & = - \int_s^t \nu \varv_r \otimes \nabla \varv_r (\nabla_y F_N) \, dr = \int_s^t \nu \varv_r \otimes \nabla \varv_r (\nabla_x F_N) \, dr  \\
&= - \int_s^t \nu \nabla \varv_r \otimes \nabla \varv_r (F_N) \, dr, 
\end{align*}
and hence that
\begin{align*}
2 \int_s^t \nu \varv_r \st \Delta \varv_r (F_N) \, dr  = - 2\int_s^t \nu \nabla \varv_r \otimes \nabla \varv_r (F_N) \, dr .
\end{align*}
Since $v \in L_T^2\bH^1$, we have $\nabla \varv_r \otimes \nabla \varv_r (F_N) \rightarrow |\nabla \varv_r |_0^2$ as $N\rightarrow \infty $ for almost all $r \in [s,t]$. Using the bound $|\nabla \varv_r \otimes \nabla \varv_r (F_N)| \leq |\nabla \varv_r |_0^2$, it follows from the  dominated convergence theorem that 
$$
\lim_{N \rightarrow \infty} 2 \int_s^t \nu \varv_r \st \Delta \varv_r (F_N) \, dr = - 2\nu \int_s^t |\nabla \varv_r |_0^2 \, dr .
$$
Using the divergence theorem again, we find
$$
\int_s^t \varv_r \otimes (u_r^i \cdot \nabla) u^i_r  (F_N) \, dr = - \int_s^t \varv_r \otimes (u^i_r)^T u_r^i  ( \nabla_y F_N)\,dr = - \int_s^t \nabla \varv_r \otimes (u^i_r)^T u_r^i  (F_N)\,dr .
$$
Using the interpolation inequality $|(u_r^i)^T u_r^i|_0 \lesssim |u_r^i|_0^{\frac{1}{2}} | u_r^i|_1^{\frac{1}{2}}$, we apply the dominated convergence theorem to get
$$
\lim_{N \rightarrow \infty} 2 \int_s^t \varv_r \st B_P(u_r^i)  (F_N) \,dr = 2 \int_s^t  B_P(u_r^i)(\varv_r) \, dr.
$$

We are now ready to finish the proof of uniqueness. 

\begin{theorem} \label{contractiveTheorem}
Let $d=2$ and assume the vector fields $\sigma_k(x) = \sigma_k$, $k\in \{1,\ldots,K\}$, are constant. Suppose that $u^1$ and $u^2$ are two solutions of \eqref{NSDiffForm} in the sense of Definition \ref{def:solution2}. Then the difference $v = u^1 - u^2$ satisfies
\begin{align}\label{eq:EnergydiffEquality}
|\varv_t|_0^2  + & 2 \int_0^t (B_P(u_r^1) - B_P(u_r^2)) ( \varv_r  ) \, dr   + 2\nu  \int_0^t  |\nabla \varv_r|_0^2 \, dr   = |\varv_0|_0^2, \quad \forall t\in [0,T].
\end{align}
Furthermore, there is a constant  $c=c(\nu,T)$ such that
\begin{equation} \label{eq:EnergydiffInequality}
|\varv_t|_0^2 + \int_0^t  |\nabla \varv_r|_0^2 \, dr \lesssim_{\nu,T} |\varv_0|_0^2 \exp \left\{ c \int_0^t |u^1_r|^2_0 | u_r^1|^2_1 \, dr \right\},\;\;\forall t\in[0,T].
\end{equation}
Therefore, there exists a unique solution $u$ of \eqref{NSDiffForm}.
\end{theorem}

\begin{remark}
The right-hand-side of \eqref{eq:EnergydiffInequality} is finite. Indeed, we have
$$
\int_0^t |u^1_r|^2_0 |  u_r^1|^2_1\, dr \leq \sup_{t \in [0,T]} |u^1_t |_0^2 \int_0^T |  u^1_r|_1^2 \, dr,
$$
which is finite since $u\in L_T^2\bH^1\cap L^{\infty}_T\bH^0$.
\end{remark}

\begin{proof}[Proof of Theorem \ref{contractiveTheorem}.]
Testing equation \eqref{eq:tensor difference} against $F_N$ and using that $\Gamma_{st}^{i,*} F_N = 0$ for $i\in \{1,2\}$, we  find 
$$
\delta \varv_{st}^{\otimes 2} (F_N)  - 2\int_s^t \left(\nu \varv_r \st \Delta \varv_r + \varv_r \st (B_P(u^1_r) - B_P(u^2_r)) \right)\,dr (F_N) = \varv_{st}^{\otimes 2, \natural}(F_N).
$$
Since the  left-hand-side is an increment of a function,  the right-hand-side $(s,t) \mapsto \varv_{st}^{\otimes 2, \natural}(F_N)$ must be as well. By virtue of Lemma \ref{tensorLemma}, we know that $\varv_{}^{\otimes 2, \natural}(F_N)$ has finite $\frac{p}{3}$-variation, which is only possible if $\varv_{st}^{\otimes 2, \natural}(F_N) = 0$.  Thus, for every $N\in \bN$,
$$
\delta \varv_{st}^{\otimes 2} (F_N)  - 2\int_s^t \left( \nu \varv_r \st \Delta \varv_r - \varv_r \st (B_P(u^1_r) - B_P(u^2_r))\right) \, dr (F_N) = 0.
$$
Passing to the limit as $N \rightarrow \infty$ in the above equality, we get
\begin{align*}
\delta (|v|_0^2)_{st}  + & 2 \int_s^t (B_P(u_r^1) - B_P(u_r^2)) ( \varv_r  )\,dr   +  2\nu \int_s^t  |\nabla \varv_r|_0^2\,dr   = 0, \;\;\forall t\in [0,T].
\end{align*}
Moreover,  using \eqref{eq:B prop},  \eqref{ineq:Lady est}, and Young's inequality (i.e,  $ab \leq \epsilon a^\frac{4}{3} + c_{\epsilon} b^{4}$, $\forall a,b,\epsilon\ge 0$, where $c_{\epsilon}$ is a constant depending only on $\epsilon$),  for every $\epsilon>0$, we have 
$$
(B_P(u^1)-B_P(u^2))(\varv)=-B_P(\varv,\varv)(u^1)\le c | \varv|_1^{\frac{3}{2}}|\varv|_0^{\frac{1}{2}}|u^1|_0^{\frac{1}{2}}|  u^1|_1^{\frac{1}{2}}
\le \epsilon|  \varv|_1^2+c_{\epsilon}|\varv|_0^{2}|u^1|_0^{2}|  u^1|_1^{2},
$$
and hence
$$
|\varv_t|_0^2 + 2\nu  \int_0^t  |\nabla \varv_r|_0^2\,dr \le  |\varv_0|_0^2 + \epsilon \int_0^t | \varv_r|_1^2\,dr + c_{\epsilon} \int_0^t |\varv_r|_0^2 |u^1_r|^2_0 | u_r^1|^2_1\,dr.
$$
Choosing $\epsilon $ small enough, we find
$$
|\varv_t|_0^2 +  \int_0^t  |\nabla \varv_r|_0^2\,dr \lesssim_{\nu}  |\varv_0|_0^2 +   \int_0^t |\varv_r|_0^2 (1+|u^1_r|^2_0 | u_r^1|^2_1)\,dr.
$$
We then complete the proof by applying  Gronwall's lemma. From the uniqueness of the velocity and the pressure recovery in Section \ref{s:pressure}, we immediately obtain the uniqueness of the associated pressure $\pi$.
\end{proof}

\subsubsection{Energy equality and continuity}

Letting $ u^1=u$ and $u^2 = 0$ in \eqref{eq:EnergydiffEquality}, where $u$ is the unique solution, we obtain the following corollary.

\begin{corollary}
Let $d=2$ and assume the vector fields $\sigma_k(x) = \sigma_k$ are constant for all $k\in \{1,\dots,K\}$. Then the unique solution  $u$  of \eqref{NSDiffForm} is in $C_T\bH^0$ and satisfies the  energy equality:
\begin{equation} \label{EnergyEqualityEq}
|u_t |_0^2 + 2  \nu \int_0^t |\nabla u_r|_0^2\,dr = |u_0|_0^2, \quad \forall t\in[0,T].
\end{equation}
\end{corollary}

\begin{proof}
We start by showing that $u$ is continuous as a mapping with values in $\bH^0$ equipped with the weak topology.
It is immediate from \eqref{RNS2} that $\lim_{s \rightarrow t} u_s(\phi) = u_t(\phi)$ for any $\phi \in \bH^3$. Moreover, since $\{ |u_s|_0 \}_{s \in [0,T] }$ is bounded, there exists a subsequence $\{ u_{s_n} \}_{n} \subset \{ u_s \}_{s \rightarrow  t}$ such that $ u_{s_n}(\phi)$ has a  limit for all $\phi \in \bH^3$. Because $\bH^3$ is dense in $\bH^0$ and weak limits are unique, we must have convergence $\lim_{s \rightarrow t} u_s(\phi) = u_t(\phi)$ for all $\phi \in \bH^0$. By virtue of the energy equality \eqref{EnergyEqualityEq}, we have that $\lim_{s \rightarrow t} |u_s|_0 = |u_t|_0$, which implies strong convergence.
\end{proof}

\begin{remark}

For constant vector fields $\sigma_k$, we have $A^{Q,i}_{st} u_s = 0$ for $i\in \{1,2\}$, and so \eqref{SystemSolutionPi} reduces to the deterministic case; that is,
$$
\pi_t = -   \int_0^tQ(u_r \cdot \nabla ) u_r\,dr.
$$
Applying  \eqref{trilinear form estimate} with $\alpha_1=1,\alpha_2=0$ and $\alpha_3=1$, we find that   $\pi\in C^{1-\textnormal{var}}([0,T];\bH^{-1}_{\perp})$.  
\end{remark}

\subsection{Stability in two spatial dimension, proof of Corollary \ref{cor:stability}}
\label{s:stab}

\begin{proof}[Proof of Corollary \ref{cor:stability}]
For $n\in\bN$, consider a sequence of  initial conditions $\{u^n_0\}_{n=1}^{\infty}\subset \bH^0$, constant vector fields $\{\sigma^{n}\}_{n=1}^\infty\subset \bR^{d\times k}$ and continuous geometric $p$-rough paths $\{\bZ^n=(Z^n,\mathbb{Z}^n)\}_{n=1}^{\infty}\in  \mathcal{C}^{p-\textnormal{var}}_g([0,T];\bR^K)$. According to Theorem \ref{existenceThm}, there exists a sequence $(u^n,\pi^n)_{n=1}^{\infty}$ of solutions to \eqref{NSDiffFormSystem} corresponding to the datum $\{(u_0^n,\sigma^n,\bZ^n)\}_{n=1}^{\infty}$. Moreover, by virtue of the energy equality  \eqref{EnergyEqualityEq}, we have
\begin{equation}\label{ineq:nrgbound in stab}
|u_t^n|_0^2+2\nu \int_0^t|\nabla u_r^n|^2\,dr=|u_0^n|^2, \;\;\forall t\in[0,T].
\end{equation}
Thus, in view of Lemma  \ref{Thm2.5} and Lemma \ref{AprioriVariation} and Remark \ref{rem:oneindexpathslocal},  we obtain 
\begin{equation}\label{ineq:pvarbound in stab}
|u^n|_{p-\textnormal{var};[0,T];\bH^{-1}}\leq c\big(|u^n_0|_0,|\sigma^n|,|Z^n|_{p-\textnormal{var};[0,T]},|\mathbb{Z}^n|_{\frac{p}{2}-\textnormal{var};[0,T]}\big),
\end{equation}
for some function $c$  that is increasing in its arguments.

Assume now that $u^n_0\to u_0$ in $\bH^0$, $\sigma^n\to\sigma$ in $\bR^{2\times K}$ and  $\bZ^n\to \bZ=(Z,\mathbb Z)$ in the rough path topology \eqref{p-var-rp} (i.e., $Z^n\to Z$ in $C^{p-\textnormal{var}}_2([0,T];\bR^K)$ and $\mathbb Z^n\to \mathbb{Z}$ in $C^{p-\textnormal{var}}_2([0,T];\bR^{K\times K})$). Then the estimates \eqref{ineq:nrgbound in stab} and \eqref{ineq:pvarbound in stab} yields a uniform (in $n$) bound for the sequence $\{u^n\}_{n=1}^{\infty}$ in $L^\infty_T\bH^0\cap L^2_T\bH^1\cap C^{p-\textnormal{var}}([0,T];\bH^{-1})$. Hence, due to Lemma \ref{lem:weakc}  there exists $u\in L^\infty_T\bH^0\cap L^2_T\bH^1\cap C^{p-\textnormal{var}}([0,T];\bH^{-1})$ such that, up to a subsequence,
$$
u^n\to u\quad\text{in}\quad L^2_T\bH^0\cap C_T\bH_w^0
$$
as $n$ tends to infinity.

Similar to the proof of Theorem \ref{existenceThm}, we may pass to the limit in the equation  and verify that $u$ solves \eqref{NSDiffFormSystem} with the datum $(u_0,\sigma,\bZ)$. Since uniqueness holds true for \eqref{NSDiffFormSystem} in two dimensions with constant vector fields, we deduce that the whole sequence $u^n$ converges to $u$ in $L^2_T\bH^0\cap C_T\bH_w^0$.

To see the convergence of $\pi^n$, we  note that since the vector fields are constant we have $A^{Q,i}_{st} u_s = 0$ for $i\in \{1,2\}$, and hence 
$$
\pi_t^n = - \int_0^t B_Q(u_r)\,dr.
$$
The convergence $\pi^n \rightarrow \pi$ in $C^{1-\textnormal{var}}([0,T]; \bH_{\perp}^{-2})$ follows since $u^n$ converges to $u$ in $L^2_T\bH^0$. Indeed,
\begin{align*}
\left| \int_s^t B_Q(u_r)( \psi)  - B_Q(u_r^n)(\psi)\,dr \right| & \leq \left| \int_s^t B_Q(u_r - u_r^n, u_r)( \psi)\,dr \right|    + \left| \int_s^t B_Q(u_r^n, u_r - u_r^n)( \psi)\,dr \right| \\
& \lesssim  \int_s^t |u_r - u_r^n|_0 |u_r|_1 |\psi |_2\,dr   +  \int_s^t |u^n_r|_1 |\psi |_2 |u_r - u_r^n|_0  \,dr
\end{align*}
for every $\psi \in \bH_{\perp}^2$, where we have used  \eqref{eq:B prop}  and \eqref{trilinear form estimate} with $\alpha_1 = \alpha_2 = 0$ and $\alpha_3 = 2$, as well as  $\alpha_1 =1$, $\alpha_2 = 1$ and $\alpha_3 = 0$. 
\end{proof}

\appendix

\section{Compact embedding results} 

The following compact embedding result is comparable to the fractional version of the  Aubin-Lions compactness result (see, e.g., \cite[Theorem 2.1]{flandoli1995martingale}).  Before we come to the embedding itself, we need to prove a simple lemma.

\begin{lemma} \label{lemmata}
If $\omega$ is a continuous control, then
$$
\lim_{a \rightarrow 0} \sup_{s \in [0,T]} \sup_{t \in [s,s+a]} \omega(s,t) = 0.
$$
\end{lemma}

\begin{proof}

Owing to superadditivity,  for any $t \in [s, s + a]$, we have $\omega(s,t) \leq \omega(s, s+a)$, and hence the claim follows once we show that
$$
\lim_{a \rightarrow 0} \sup_{s \in [0,T]}  \omega(s,s+a) = 0 .
$$
Suppose, by contradiction, that there exists an $\epsilon > 0$ and a sequence $\{(s_n,a_n)\}_{n=1}^{\infty} \subset [0,T] \times [0,1]$ such that $\lim_{n\rightarrow \infty}a_n = 0$ and
$$
\omega(s_n, s_n + a_n)  > \epsilon, \quad \forall n\in \mathbf{N}.
$$
Since $[0,T]$ is compact, there exists an $s \in [0,T]$ and a subsequence $\{(s_{n_k}, a_{n_k})\}_{k=1}^{\infty} \subset \{(s_n, a_n)\}_{n=1}^{\infty}$ converging to $(s,0)$. By the  continuity of the control $\omega$, we have
$$
\epsilon \leq \lim_{k \rightarrow \infty} \omega(s_{n_k}, s_{n_k} + a_{n_k}) = \omega(s,s) = 0,
$$
which is a contradiction. 
\end{proof}

\begin{lemma} \label{CompactnessLemma}
Let  $\omega$ and $\varpi$  be a controls on $[0,T]$ and $L,\kappa >0$. Let
$$
X=L_T^2\bH^1 \cap \left\{ g  \in C_T\bH^{-1}  : |\delta g_{st}|_{-1} \leq \omega(s,t)^{\kappa}, \;\forall (s, t)  \in \Delta_{T} \textnormal{  with } \varpi(s,t) \leq L   \right\}
$$
be endowed with the norm
$$
|g|_X=|g|_{L_T^2\bH^1}+\sup_{t\in [0,T]}|g_t|_{-1}+\sup\left\{\frac{ |\delta g_{st}|_{-1}}{\omega(s,t)^{\kappa}}: (s,t)\in \Delta_T \;\textnormal{ s.t.} \;\varpi(s,t) \leq L\right\}.
$$
Then $X$ is compactly embedded into $C_T\bH^{-1}$ and $L^2_T \bH^0$.
\end{lemma}
\begin{proof}
For each $a \in (0,L] $ and every $g\in L_T^2\bH^{-1}$, let us define the function $J_ag:[0,T]\rightarrow \bH^{-1}$ by
$$
J_a g_s =  \frac{1}{a} \int_{s}^{s+a }  g_t\, dt=\frac{1}{a} \int_{0}^{a }  g_{s+t}\, dt,
$$		
where we extend $g$ to $\bR_+$ by letting $g = g_T$ outside $[0,T]$. Clearly, $s\mapsto J_ag_s$ is continuous from $[0,T]$ into $\bH^{-1}$; that is,  $J_a$ is a well-defined  map from $L_T^2\bH^{-1}$ to $C_T\bH^{-1}$.  Moreover, using H\"older's inequality, for $i\in \{-1,1\}$, we find
\begin{equation}\label{ineq:Japointbound}
|J_ag_s|_{i}\le \frac{1}{a} \int_{0}^{a }  |g_{s+t}|_{i}\, dt\le \frac{1}{\sqrt{a}} \left(\int_{0}^{a}  |g_{s+t}|_{i}^2\, dt\right)^{\frac{1}{2}},
\end{equation}
which implies
$$
\int_0^T 	|J_ag_s|_{i}^2ds\le \frac{1}{a} \int_0^T\int_{0}^{a }  |g_{s+t}|_{i}^2\, dtds= \int_0^T|g_{t}|_{i}^2\, dt,
$$
and hence $|J_ag|_{L^2_T\bH^{i}}\le |g|_{L^2_T\bH^{i}}$; that is, $J_a: L^2_T\bH^i\rightarrow \bH^i$ is a bounded operator for $i\in \{-1,1\}$.

Let us show that $J_ag \rightarrow g$ in $C_T\bH^{-1}$ as $a \rightarrow 0$ uniformly with respect to $X$. For each $s\in[0,T]$, $g\in X$ ,  we have
\begin{align*}
|J_a g_s - g_s|_{-1} & = \frac{1}{a} \left| \int_{s}^{s+a} g_t\, dt - \int_{s}^{s+a} g_s\, dt \right|_{-1}   \leq  \frac{1}{a}  \int_{s}^{s+a}\left| g_t - g_s \right|_{-1}\, dt \\
&  \leq \frac{1}{a} \int_{s}^{s+a} \omega(s,t)^{\kappa}\, dt \leq \sup_{t \in [s, s+a]} \omega(s,t)^{\kappa},
\end{align*}
which converges uniformly in $s$ to 0 as $a \rightarrow 0$ by Lemma \ref{lemmata}.

Let $\clG$ be a bounded subset of  $L_T^2\bH^1$, with norm bound denoted $N_0$.  Using H\"older's inequality, for all $s, t \in [0,T]$ and $g\in \clG$, we obtain
\begin{equation}
|J_ag_t - J_a g_{s} |_{-1} =  \frac{1}{a} \left| \int_{t+ a}^{s +a} g_r\, dr - \int_{s}^t  g_t\, dt \right|_{-1} \leq \frac{2}{a}N_0 \sqrt{|s - \bar{s}|},
\end{equation}
and hence  for a fixed $a$, $J_a\clG$ is uniformly equicontinuous $\bH^ {-1}$.  Owing to \eqref{ineq:Japointbound},  for each $s\in [0,T]$, we have that $|J_ag_s|_1\le \frac{1}{\sqrt{a}}N_0$, and hence  for a fixed $a$, $J_a\clG$ is pointwise bounded in $\bH^1$. Since $\bH^1$ is compactly embedded in $\bH^{-1}$,   for a fixed $a$, $J_a\clG$ is  pointwise relatively compact in $\bH^{-1}$. Therefore, by the generalized Arzel\`a--Ascoli theorem $J_a \clG$ is relatively compact in $C_T\bH^{-1}$.

To conclude the proof, let $\{g^n\}_{n=1}^{\infty}$ be a bounded sequence in $X$.
In particular, by Banach-Alaoglu, there exists a subsequence  $\{g^{n_k}\}_{k=1}^{\infty}$ that   converges in the  weak*-topology of $L_T^2\bH^1$ to some $g \in L_T^2\bH^1$. We can reduce to the case $g=0$, and hence the proof of the compact embedding of $X$ in $C_T\bH^{-1}$ is complete if we can show that $|g^{n_k}|_{C_T\bH^{-1}}\rightarrow 0$ as $k \rightarrow \infty$. 

To this end,  for any fixed $a \in [0,L] $, by the above Arzel\`a--Ascoli argument, $\{J_a g^{n_k}\}_{k=1}^{\infty}$ has a convergent subsequence in $C_T\bH^{-1}$, which we also denote by $\{ J_a g^{n_k} \}_{k=1}^{\infty}$. We note  that this subsequence may depend on $a$. Combining this with the fact that $g^{n_k}\rightarrow 0$ in the weak*-topology of $L_T^2\bH^1$, we see that for any $f \otimes \phi \in C_T \otimes \bH^{1}$, we have
$$
\lim_{k \rightarrow \infty} \int_0^T J_a g^{n_k}_r( \phi) f_r\,dr = \lim_{k \rightarrow \infty} \int_0^T g^{n_k}_r( \phi) J_a^*f_r\,dr =0 ,
$$
so that $\lim_{k \rightarrow \infty} J_a g^{n_k} = 0$ in $C_T\bH^{-1}$. Since all subsequences converges to the same limit, this means the full sequence converges. For any $a\in (0,L]$
$$
|g^{n_k} |_{C_T\bH^{-1}} \leq |J_a g^{n_k} |_{C_T\bH^{-1}} + |J_a g^{n_k}   - g^{n_k} |_{C_T\bH^{-1}} \leq |J_a g^{n_k} |_{C_T\bH^{-1}} +  \sup_{s \in [0,T]} \sup_{t \in [s, s+a]} \omega(s,t)^{\kappa} .
$$
Letting  $k \rightarrow \infty$ first and then $a \rightarrow 0$, we find that $|g^{n_k}|_{C_T\bH^{-1}}\rightarrow 0$ as $k \rightarrow \infty$, which shows that $X$ is compactly embedded in $C_T\bH^{-1}$.

Let us now show that the set $X$ is  compactly embedded in $L_T^2\bH^0$. Using  Young's inequality, for $h \in \bH^1$ and any $\epsilon > 0$,
$$
|h|^2_0 = h(h) \leq |h|_{-1} |h|_1 \leq C_{\epsilon} |h|_{-1}^2 + \epsilon |h|_1^2
$$
for some appropriate constant $C_{\epsilon}>0$. Consequently, proceeding with the same sequence above, we find 
$$
|g^{n_k}|_{L_T^2 \bH^0}^2 \leq C_{\epsilon} |g^{n_k}|_{L_T^2 \bH^{-1}}^2 + \epsilon |g^{n_k}|_{L_T^2 \bH^{1}}^2 \leq C_{\epsilon} |g^{n_k}|_{C_T \bH^{-1}}^2 + \epsilon \sup_{n\in\bN}|g^{n}|_{L_T^2 \bH^{1}}^2 .
$$
Letting  $k \rightarrow \infty$ first, we have 
$$
\lim_{n \rightarrow \infty} |g^n|_{L_T^2 \bH^0}^2 \leq \epsilon \sup_{n\in \bN} |g^n |_{L_T^2 \bH^{1}}^2,
$$
and then letting $\epsilon \rightarrow 0$, we conclude the proof.
\end{proof}

Denote by $C_T\bH^{0}_w$ the space of  $\bH^0$-valued weakly continuous functions on $[0,T]$.

\begin{lemma}\label{lem:weakc}
Let  $\omega$ and $\varpi$  be controls on $[0,T]$ and $L,\kappa >0$. Let
$$
Y=L_T^{\infty}\bH^0 \cap \left\{ g  \in C_T\bH^{-1}  : |\delta g_{st}|_{-1} \leq \omega(s,t)^{\kappa}, \;\forall (s, t)  \in \Delta_{T} \textnormal{  with } \varpi(s,t) \leq L   \right\},
$$
be endowed with the norm
$$
|g|_Y=|g|_{L_T^\infty\bH^0}+\sup_{t\in [0,T]}|g_t|_{-1}+\sup\left\{\frac{ |\delta g_{st}|_{-1}}{\omega(s,t)^{\kappa}}: (s,t)\in \Delta_T \;\textnormal{ s.t.} \;\varpi(s,t) \leq L\right\}.
$$
Then $Y$ is compactly embedded into  $C_T\bH^{0}_w$.
\end{lemma}

\begin{proof}
Let $g\in Y$ be arbitrarily chosen. First, we will show that for all $\varphi\in \bH^0$, the mapping
\begin{equation}\label{eq:cont}
t\mapsto\langle g_t,\varphi\rangle \in C_T\bR.
\end{equation}
To this end, we observe that since $g\in L^\infty_T \bH^0$, it follows that there exists $R>0$ such that $g_t\in B_R$ for all $t\in[0,T]$, where $B_R\subset \bH^0$ is a ball of radius $R$.
Let  $\{h_n\}_{n=1}^{\infty}\subset \bH^1$ be a family whose finite linear combinations are dense in $\bH^{0}$. Then
\begin{align}\label{eqcon}
\left| \langle g_t, \varphi\rangle -  \langle g_s,\varphi\rangle \right| &\leq 
\left| \left\langle g_t-g_s, \sum_{n\leq M} \beta_n h_n  \right\rangle \right| + 
\left| \left\langle g_t-g_s,\varphi - \sum_{n\leq M} \beta_n h_n ,  \right\rangle \right| \nonumber\\ 
&\leq 
\left| \left\langle g_t-g_s, \sum_{n \leq M} \beta_n h_n , \right\rangle \right| 
+ 2R \left| \varphi - \sum_{n \leq M} \beta_n h_n \right|_{0}\nonumber\\
&\leq c(M) \omega(s,t)^\kappa+ 2R \left|\varphi - \sum_{n \leq M} \beta_n h_n \right|_{0}, 
\end{align}
where the last term can be made small uniformly for all $s,t \in [0,T]$ by taking  $M$ large enough and suitable $\{\beta_{{m}}\}_{m=1}^M$. Hence, \eqref{eq:cont} follows. The compactness of the embedding follows from the  generalized Arzel\` a--Ascoli theorem. Indeed, the ball $B_R$ is relatively weakly 
compact, and the desired equicontinuity follows from  \eqref{eqcon}.
\end{proof}

\section{Sewing lemma}

The following lemma, referred to as the \textit{sewing lemma}, lies at the very foundation of the theory of rough paths. The proof is a straightforward modification of Lemma 2.1 in \cite{DeGuHoTi16}. See, also,  Lemma 4.2 in \cite{FrHa14}.

\begin{lemma}[c.f. Lemma 2.1 in\cite{DeGuHoTi16} and Lemma 4.2 in \cite{FrHa14}] \label{sewingLemma}
Let $I$  be a subinterval of $[0,T]$, $E$ be a Banach space  and  $\zeta \in [0,1)$. Let  $\omega$ and $\varpi$ be controls on $I$ and $L > 0$. Assume that  $h:\Delta_I\rightarrow  E$  is such  that for all $(s,u,t)\in \Delta^{(2)}_I $ with $\varpi(s,t) \leq L$,
$$
|\delta h_{sut}|\le \omega(s,t)^{\frac{1}{\zeta}}.
$$
Then there exists a unique path $\clI h:I\rightarrow E$ with $\clI h_0=0$ such that  $\Lambda h:=h-\delta \clI h\in C_{2,\varpi, L}^{\zeta-\textnormal{var}}(I;E)$. Moreover, there exists a universal constant $C_{\zeta}>0$ such that for all $(s,t)\in \Delta_I $ with $\varpi(s,t) \leq L$,
\begin{equation}\label{ineq:sewing estimate}
|(\Lambda   h)_{st}|\le C_{\zeta}\omega(s,t)^{\frac{1}{\zeta}}.
\end{equation}
Furthermore, if   $h\in  C_{2,\varpi,L}^{p-\textnormal{var}}(I;E)$ for some $p\ge \zeta$, then  $\clI h\in C^{p-\textnormal{var}}(I;E)$.
\end{lemma}
The following corollary is immediate since $\clI h$ is a path with $\clI h_0=0$, and hence vanishes if $\clI h\in C^{p-\textnormal{var}}(I;E)$ for $p>1$.
\begin{corollary}
Assume the hypothesis of Lemma \ref{sewingLemma}. If   $h\in  C_{2,\varpi,L}^{p-\textnormal{var}}(I;E)$ for some $p<1$, then for all $(s,t)\in \Delta_I $
 with $\varpi(s,t) \leq L$,
$$
| h_{st}|\le C_{\zeta}\omega(s,t)^{\frac{1}{\zeta}}.
$$
\end{corollary}

\bibliographystyle{unsrt}
\bibliography{bibliography}

\begin{thebibliography}{10}

\bibitem{Ly98}
T.~J. Lyons.
\newblock Differential equations driven by rough signals.
\newblock {\em Revista Matem{\'a}tica Iberoamericana}, 14(2):215--310, 1998.

\bibitem{CaFr09}
M.~Caruana and P.~Friz.
\newblock Partial differential equations driven by rough paths.
\newblock {\em J. Differential Equations}, 247(1):140--173, 2009.

\bibitem{CaFrOb11}
M.~Caruana, P.~K. Friz, and H.~Oberhauser.
\newblock A (rough) pathwise approach to a class of non-linear stochastic
  partial differential equations.
\newblock In {\em Annales de l'Institut Henri Poincare (C) Non Linear
  Analysis}, volume~28, pages 27--46. Elsevier, 2011.

\bibitem{GuTi10}
M.~Gubinelli and S.~Tindel.
\newblock Rough evolution equations.
\newblock {\em The Annals of Probability}, 38(1):1--75, 2010.

\bibitem{DeGuTi12}
A.~Deya, M.~Gubinelli, and S.~Tindel.
\newblock Non-linear rough heat equations.
\newblock {\em Probability Theory and Related Fields}, 153(1):97--147, 2012.

\bibitem{GuLeTi06}
M.~Gubinelli, A.~Lejay, and S.~Tindel.
\newblock Young integrals and {SPDE}s.
\newblock {\em Potential Analysis}, 25(4):307--326, 2006.

\bibitem{GuImPe15}
M.~Gubinelli, P.~Imkeller, and N.~Perkowski.
\newblock Paracontrolled distributions and singular pdes.
\newblock In {\em Forum of Mathematics, Pi}, volume~3, page~e6. Cambridge Univ
  Press, 2015.

\bibitem{Ha14}
M.~Hairer.
\newblock A theory of regularity structures.
\newblock {\em Invent. Math.}, 198(2):269--504, 2014.

\bibitem{BaGu15}
I.~Bailleul and M.~Gubinelli.
\newblock Unbounded rough drivers.
\newblock {\em arXiv preprint arXiv:1501.02074}, 2015.

\bibitem{DeGuHoTi16}
A.~Deya, M.~Gubinelli, M.~Hofmanov{\'a}, and S.~Tindel.
\newblock A priori estimates for rough {PDE}s with application to rough
  conservation laws.
\newblock {\em arXiv preprint arXiv:1604.00437}, 2016.

\bibitem{2017arXiv170707470H}
A.~{Hocquet} and M.~{Hofmanov{\'a}}.
\newblock {An energy method for rough partial differential equations}.
\newblock {\em ArXiv e-prints}, July 2017.

\bibitem{DeGuHoTi16a}
A.~Deya, M.~Gubinelli, M.~{Hofmanov{\'a}}, and S.~Tindel.
\newblock One-dimensional reflected rough differential equations.
\newblock {\em arXiv preprint arXiv:1610.07481}, 2016.

\bibitem{young1936inequality}
L.~C. Young.
\newblock An inequality of the {H}{\"o}lder type, connected with {S}tieltjes
  integration.
\newblock {\em Acta Mathematica}, 67(1):251--282, 1936.

\bibitem{cotter2017stochastic}
Colin~J Cotter, Georg~A Gottwald, and Darryl~D Holm.
\newblock Stochastic partial differential fluid equations as a diffusive limit
  of deterministic {L}agrangian multi-time dynamics.
\newblock {\em Proc. R. Soc. A}, 473(2205):20170388, 2017.

\bibitem{kraichnan1968small}
R.~H. Kraichnan.
\newblock Small-scale structure of a scalar field convected by turbulence.
\newblock {\em The Physics of Fluids}, 11(5):945--953, 1968.

\bibitem{gawedzki1996university}
K.~Gawedzki and A.~Kupiainen.
\newblock University in turbulence: An exactly solvable model.
\newblock {\em Low-dimensional models in statistical physics and quantum field
  theory}, pages 71--105, 1996.

\bibitem{gawedzki2000phase}
K.~Gaw{\c e}dzki and M.~Vergassola.
\newblock Phase transition in the passive scalar advection.
\newblock {\em Physica D: Nonlinear Phenomena}, 138(1):63--90, 2000.

\bibitem{brzezniak1991stochastic}
Z.~Brze{\'z}niak, M.~Capi{\'n}ski, and F.~Flandoli.
\newblock Stochastic partial differential equations and turbulence.
\newblock {\em Mathematical Models and Methods in Applied Sciences},
  1(01):41--59, 1991.

\bibitem{brzezniak1992stochastic}
Z.~Brze{\'z}niak, M.~Capi{\'n}ski, and F.~Flandoli.
\newblock Stochastic {N}avier-{S}tokes equations with multiplicative noise.
\newblock {\em Stochastic Analysis and Applications}, 10(5):523--532, 1992.

\bibitem{mikulevicius2004stochastic}
R.~Mikulevicius and B.~L. Rozovskii.
\newblock Stochastic {N}avier--{S}tokes equations for turbulent flows.
\newblock {\em SIAM Journal on Mathematical Analysis}, 35(5):1250--1310, 2004.

\bibitem{mikulevicius2005global}
R.~Mikulevicius and B.~L. Rozovskii.
\newblock Global {L}2-solutions of stochastic {N}avier--{S}tokes equations.
\newblock {\em The Annals of Probability}, 33(1):137--176, 2005.

\bibitem{flandoli1995martingale}
F.~Flandoli and D.~Gatarek.
\newblock Martingale and stationary solutions for stochastic {N}avier-{S}tokes
  equations.
\newblock {\em Probability Theory and Related Fields}, 102(3):367--391, 1995.

\bibitem{CM11}
I.~Chueshov and A.~Millet.
\newblock Stochastic two-dimensional hydrodynamical systems:\ {W}ong-{Z}akai
  approximation and support theorem.
\newblock {\em Stoch. Anal. Appl.}, 29(4):570--611, 2011.

\bibitem{CM10}
I.~Chueshov and A.~Millet.
\newblock Stochastic 2{D} hydrodynamical type systems: well posedness and large
  deviations.
\newblock {\em Appl. Math. Optim.}, 61(3):379--420, 2010.

\bibitem{mikulevicius2003cauchy}
R~Mikulevicius.
\newblock On the {C}auchy problem for the stochastic {S}tokes equations.
\newblock {\em SIAM Journal on Mathematical Analysis}, 34(1):121--141, 2002.

\bibitem{mikulevicius2001note}
R~Mikulevicius, B~Rozovskii, et~al.
\newblock A note on {K}rylov's $ l\_p $-theory for systems of {SPDE}s.
\newblock {\em Electronic Journal of Probability}, 6, 2001.

\bibitem{RT83}
R.~Temam.
\newblock {\em {N}avier-{S}tokes equations and nonlinear functional analysis},
  volume~41 of {\em CBMS-NSF Regional Conference Series in Applied
  Mathematics}.
\newblock Society for Industrial and Applied Mathematics (SIAM), Philadelphia,
  PA, 1983.

\bibitem{hebey2000nonlinear}
Emmanuel Hebey.
\newblock {\em Nonlinear analysis on manifolds: {S}obolev spaces and
  inequalities}, volume~5.
\newblock American Mathematical Soc., 2000.

\bibitem{FrVi10}
P.~K. Friz and N.~B. Victoir.
\newblock {\em Multidimensional stochastic processes as rough paths: theory and
  applications}, volume 120.
\newblock Cambridge University Press, 2010.

\bibitem{MR2314753}
T.~J. Lyons, M.~Caruana, and T.~L{\'e}vy.
\newblock {\em Differential equations driven by rough paths}, volume 1908 of
  {\em Lecture Notes in Mathematics}.
\newblock Springer, Berlin, 2007.
\newblock Lectures from the 34th Summer School on Probability Theory held in
  Saint-Flour, July 6--24, 2004, With an introduction concerning the Summer
  School by Jean Picard.

\bibitem{FrHa14}
P.~K. Friz and M.~Hairer.
\newblock {\em A course on rough paths: with an introduction to regularity
  structures}.
\newblock Universitext. Springer, Cham, 2014.

\bibitem{FrOb14}
P.~K. Friz and H.~Oberhauser.
\newblock Rough path stability of (semi-) linear spdes.
\newblock {\em Probability Theory and Related Fields}, 158(1-2):401--434, 2014.

\bibitem{DiFrOb14}
J.~Diehl, P.~K. Friz, and H.~Oberhauser.
\newblock Regularity theory for rough partial differential equations and
  parabolic comparison revisited.
\newblock In {\em Stochastic analysis and applications 2014}, volume 100 of
  {\em Springer Proc. Math. Stat.}, pages 203--238. Springer, Cham, 2014.

\end{thebibliography}
\end{document}